\documentclass[12pt,a4paper]{amsart}
\usepackage[utf8]{inputenc}
\usepackage[T1]{fontenc}
\usepackage{amsmath,amsthm,amsfonts,amssymb,mathtools}
\usepackage{xcolor}
\usepackage[hypertexnames=false,colorlinks=true,linkcolor=teal,citecolor=blue]{hyperref}
\newcommand{\supp}{\operatorname{supp}}

\newcommand{\bbfR}{\mathbb{R}}
\newcommand{\vf}{\varphi}
\newcommand{\bbfN}{\mathbb{N}}
\newcommand{\ud}{\,\mathrm d}

\newtheorem{theorem}{Theorem}

\newtheorem{lemma}[theorem]{Lemma}

\theoremstyle{definition}

\theoremstyle{remark}
\newtheorem{remark}[theorem]{Remark}
\numberwithin{equation}{section}
\numberwithin{theorem}{section}
\mathtoolsset{showonlyrefs=true}

\author[X. Zheng]{Xiaoxin Zheng}
\address[X. Zheng]{\small 
School of Mathematics and Systems Science, Beihang University, Beijing 100191, 
People's Republic of China}
\email{xiaoxinzheng@buaa.edu.cn}

\author[D. Pilarczyk]{Dominika Pilarczyk}
\address[D. Pilarczyk]{\small Wydzia{\l} Matematyki, Politechnika Wroc{\l}awska,
Wybrze\.ze Wyspia\'nskiego 37, Wroc{\l}aw, Poland} 
\email{dominika.pilarczyk@pwr.edu.pl}

\author[G. Karch]{Grzegorz Karch}
\address[G. Karch]{\small 
 Instytut Matematyczny, Uniwersytet Wroc{\l}awski,
 pl. Grunwaldzki 2/4, \hbox{50-384} Wroc{\l}aw, Poland}
\email{Grzegorz.Karch@math.uni.wroc.pl}

\author[M. Cannone]{Marco Cannone}
\address[M. Cannone]{\small 
Laboratoire d'Analyse et de Mathématiques Appliquées,
Universit\'e Gustave Eiffel,
5 boulevard Descartes,
Bâtiment Copernic,
77420 Champs-sur-Marne, France.}
\email{marco.cannone@univ-eiffel.fr}

\title[Large self-similar solutions]{Large self-similar solutions \\ of  a convection--diffusion equation}

\begin{document}
\begin{abstract}
We construct self-similar solutions to a nonlinear convection--diffusion equation posed on the whole Euclidean space, for a class of homogeneous initial data. The corresponding self-similar profiles solve a nonlinear elliptic equation on $\mathbb{R}^N$. We obtain these profiles as limits of solutions to Dirichlet problems on bounded domains with truncated nonlinearities, using a priori estimates that are uniform with respect to both the domain size and the truncation level.

Our approach is inspired by recent constructions of large self-similar solutions for the Navier--Stokes system. The convection--diffusion equation considered here shares some structural features with the Navier--Stokes equations, but its scalar parabolic nature allows us to exploit tools such as the maximum principle and thereby obtain a more direct construction of self-similar solutions.
\end{abstract}

\keywords{Convection--diffusion equation; self-similar solution; elliptic equation; homogeneous initial data.}

\subjclass[2020]{Primary 35C06; Secondary 35K58, 35K60, 35J61, 76R99}

\date{}
\maketitle

\tableofcontents

\section{Introduction}
\subsection{Statement of the problem and a brief description of the main result}

Our goal is to construct self-similar solutions to the convection--diffusion equation
\begin{equation}\label{eq}
u_t - \Delta u + \partial_{x_1}\left(|u|^{q-1}u\right) = 0, \qquad x \in \mathbb{R}^N, \; t > 0,
\end{equation}
for $q > 1 + {1}/{N}$. Equation~\eqref{eq} is invariant under the scaling transformation
\begin{equation}\label{scaling:intro}
u_\lambda(x,t) \equiv \lambda^{\frac{1}{q-1}} u(\lambda x, \lambda^2 t) \qquad \text{for all } \lambda > 0.
\end{equation}
That is, if $u(x,t)$ solves \eqref{eq}, then so does $u_\lambda(x,t)$.
A solution to equation \eqref{eq} is called \textit{self-similar} if it remains invariant under this scaling:
\begin{equation}\label{u:ul}
u(x,t) = \lambda^{\frac{1}{q-1}} u(\lambda x, \lambda^2 t) \qquad \text{for all } x \in \mathbb{R}^N, \; t > 0, \text{ and } \lambda > 0.
\end{equation}
Setting $\lambda = 1/\sqrt{t}$ in \eqref{u:ul} shows that every self-similar solution takes the form
\begin{equation}\label{self_similar}
u(x,t) = t^{-\frac{1}{2(q-1)}} U\!\left(\frac{x}{\sqrt{t}}\right) \qquad \text{with } U(x) = u(x,1).
\end{equation}
If a self-similar solution \eqref{self_similar} admits an initial trace $u_0(x)$, this trace must be homogeneous of degree 
$-{1}/({q-1})$. Accordingly, we construct self-similar solutions \eqref{self_similar} to the Cauchy problem for equation \eqref{eq} with homogeneous initial data of the form
\begin{equation}\label{u0:intro}
u_0(x)=\frac{C(x)}{|x|^{\frac{1}{q-1}}}\quad 
\text{with}\quad  
C\in L^\infty (\bbfR^N)\;\; \text{such that} \;\; C(\lambda x)=C(x) \;\; \text{for all}\;\; \lambda>0.
\end{equation}
 For a nonzero angular part $C$, the function $u_0$ is locally integrable near the origin precisely when ${1}/({q-1}) < N$, which is equivalent to $q > 1 + {1}/{N}$.

Substituting expression \eqref{self_similar} into equation \eqref{eq} shows that the \textit{self-similar profile} $U(x) = u(x,1)$ satisfies the elliptic equation
\begin{equation}\label{ell_U_q:intro}
-\Delta U - \tfrac{1}{2} x \cdot \nabla U - \tfrac{1}{2(q-1)} U + \partial_{x_1}\left(|U|^{q-1} U\right) = 0, \qquad x \in \mathbb{R}^N.
\end{equation}
Our main result, Theorem~\ref{main}, establishes that equation \eqref{ell_U_q:intro} admits solutions $U(x)$ such that the rescaled function \eqref{self_similar} solves the Cauchy problem for \eqref{eq} for every homogeneous initial datum $u_0$ given by \eqref{u0:intro}. We refer to Section~\ref{section-main-results} for precise statements.

The methods developed in this paper extend naturally 
(as in other works; see, e.g.,~\cite{EZ91,EVZ93,EVZ94,EZ97})
to the more general equation
\begin{equation*}
u_t - \Delta u + \bar{a} \cdot \nabla\bigl(|u|^{q-1}u\bigr) = 0, \qquad x \in \mathbb{R}^N, \; t > 0,
\end{equation*}
where $\bar{a} \in \mathbb{R}^N \setminus \{0\}$ is a constant vector. 
Indeed, by rotating and rescaling the spatial variables, we may assume without loss of generality that $\bar{a} = (1, 0, \ldots, 0)$. 

\subsection{Related results}

In his influential monograph \cite{Barenblatt}, G.~I.~Barenblatt emphasizes the broad importance of scaling, or power-law, relations in science and engineering. Such relations should not be regarded merely as convenient empirical approximations of experimental data, since in many applications they reflect an underlying structural property of the governing phenomena. 
In particular, scaling laws often reveal self-similarity: after an appropriate rescaling of space, time, or other relevant variables, the phenomenon preserves the same essential form.
The monograph \cite{Barenblatt} presents numerous examples of scaling laws, self-similar solutions, and intermediate asymptotic behavior. A more recent treatment of self-similar solutions to partial differential equations can be found in the monograph 
by Giga, {\it et al.} \cite{GGS10}. We now discuss those results on scaling laws and self-similar solutions that are most closely related to the present paper.

Aguirre, Escobedo, and Zuazua \cite{AEZ90} constructed a family of integrable 
self-similar solutions invariant under the scaling \eqref{scaling:intro} to equation~\eqref{eq} in the critical case
\(
q=1+{1}/{N}.
\)
The special role of this exponent follows directly from
the profile equation. Let $U(x)=u(x,1)$ be a sufficiently regular
solution of the elliptic equation \eqref{ell_U_q:intro}. Assume that $U$ and its derivatives
decay fast enough at infinity to justify integration by parts.
Integrating the elliptic equation \eqref{ell_U_q:intro} over $\bbfR^N$, we obtain
\begin{equation*}
\frac{N}{2}\int_{\bbfR^N}U(x)\,\ud x
=
\frac{1}{2(q-1)}\int_{\bbfR^N}U(x)\,\ud x.
\end{equation*}
Thus, if $q\neq 1+1/N$, every integrable profile satisfying these
assumptions has zero mass. In particular, a nonnegative, nonzero,
integrable self-similar profile cannot exist outside the critical case.

When $q=1+1/N$, the nonnegative solutions constructed in the work \cite{AEZ90}
are indexed by their mass $M>0$ and have the initial trace
\(
u(\cdot,0)=M\delta_0,
\)
where $\delta_0$ is the Dirac measure at the origin. These solutions are
called source-type solutions and they resemble the Gaussian heat
kernel which solves the heat equation. They are also
important in the study of long-time behavior of other solutions to equation \eqref{eq}. 
Escobedo and Zuazua \cite{EZ91}
proved that the self-similar source-type
solution with mass $M$ gives the leading large-time profile of the
solution to equation \eqref{eq} with integrable initial data of the same mass.

The situation is different in the supercritical range $q>1+1/N$. As
explained in Remark~\ref{rem:Lp:self} below, a nontrivial profile cannot
belong to $L^p(\bbfR^N)$ when
\(
1<p\leq N(q-1).
\)
It is useful to begin the discussion of the construction of non-integrable self-similar profiles by recalling the analogous theory for the incompressible Navier--Stokes equations in $\bbfR^3$. 

This system of equations for the vector velocity $v$ and the scalar pressure $p$ has the form
\begin{equation}\label{NS:intro}
\partial_t v+(v\cdot\nabla)v-\Delta v+\nabla p=0,
\qquad
\nabla\cdot v=0,
\end{equation}
and is invariant under the scaling
\[
v_\lambda(x,t)=\lambda v(\lambda x,\lambda^2t),
\qquad
p_\lambda(x,t)=\lambda^2p(\lambda x,\lambda^2t).
\]
A forward self-similar solution satisfies $v_\lambda=v$ for every
$\lambda>0$ and therefore has the form
\[
v(x,t)=\frac{1}{\sqrt t}V\left(\frac{x}{\sqrt{t}}\right).
\]
Its initial velocity is homogeneous of degree $-1$ and it is not expected from  nonzero field
with this homogeneity to belong to $L^3(\bbfR^3)$.
 Such a field may nevertheless belong to a larger
scale-invariant space, such as weak $L^3$, a critical Morrey space, or
a homogeneous Besov space. 
Such self-similar solutions to the Navier--Stokes system \eqref{NS:intro} have been constructed by the following approach.
For small initial data, the Navier--Stokes equations can be solved
through the mild formulation
\begin{equation}\label{duhamel:intro}
v(t)
=
e^{t\Delta}v_0
-
\int_0^t
e^{(t-s)\Delta}\mathbb{P}\nabla\cdot(v\otimes v)(s)\,\ud s,
\end{equation}
where $\mathbb{P}$ is the Helmholtz--Leray projection. In a suitable
scale-invariant function space, the integral term defines a bounded
bilinear map. If the critical norm of $v_0$ is sufficiently small, a
contraction argument gives a unique global mild solution in the
corresponding solution class. If, in addition, $v_0$ is homogeneous of
degree $-1$, then $v_\lambda$ has the same initial data as $v$ and
belongs to the same uniqueness class. Uniqueness therefore implies
$v_\lambda=v$, and the global mild solution is self-similar. This
argument explains the two main features of the small-data theory:
smallness gives a contraction, and uniqueness transfers the scaling of
the initial data to the solution.

This was applied in several different critical spaces.
Giga and Miyakawa studied the vorticity equation with measure-valued
initial vorticity in Morrey spaces \cite{GM89}. Cannone and Planchon
constructed self-similar solutions in critical Besov-type spaces under
a smallness condition \cite{CP96}. 
The paper \cite{CK04} presents a particularly simple Fourier-transform approach to the construction of self-similar solutions to the Navier--Stokes system \eqref{NS:intro} for sufficiently small homogeneous initial data.
The survey \cite{C04} explains the harmonic-analysis
methods behind this perturbative theory and the construction of self-similar solutions. 
Brandolese later
derived precise far-field asymptotic formulas for small self-similar
profiles \cite{B09}.

A completely analogous approach can be employed to construct self-similar solutions to equation \eqref{eq} in the supercritical range $q>1+{1}/{N}$, under a suitable smallness assumption on the homogeneous initial data. 
The existence of such small self-similar solutions, together with related generalizations and their role in the analysis of the long-time asymptotic behavior of more general solutions, has been studied, for example, in \cite{K99,DZ99,HKNS06}.

A natural open problem was whether self-similar solutions to the Navier--Stokes system \eqref{NS:intro}, or to the convection--diffusion equation \eqref{eq}, could be constructed for homogeneous initial data without imposing a smallness condition. In this large-data regime, the perturbative approach described above is no longer directly applicable. In particular, the nonlinear map arising from the Duhamel representation \eqref{duhamel:intro} is not known to admit a fixed point for large initial data. Moreover, uniqueness is generally not available in the natural classes of weak solutions.

A major breakthrough was achieved by Jia and \v{S}ver\'{a}k \cite{JS14}, who constructed large forward self-similar solutions to the Navier--Stokes equations without imposing a smallness condition on the homogeneous initial data. 
Their approach combines local-in-space estimates for local-Leray solutions near the initial time, \textit{a priori} estimates for the associated elliptic profile problem, and a Leray--Schauder degree argument.
This large-data theory for the Navier--Stokes equations has subsequently been extended in several directions. Tsai \cite{T14} constructed large forward discretely self-similar solutions, and Bradshaw and Tsai \cite{BT17} further developed this theory. Chae and Wolf \cite{CW18} considered discretely self-similar initial data in $L^2_{\mathrm{loc}}(\bbfR^3)$. Korobkov and Tsai \cite{KorTsai16} introduced a simple energy-based argument to construct forward self-similar solutions corresponding to arbitrarily large self-similar data in the three-dimensional half-space. A related idea was later used in the works \cite{BMZ19,Lai24,hou2026forwardselfsimilarsolutions2d} in the study of large forward self-similar solutions to the fractional Navier--Stokes equations for a range of fractional diffusion exponents that includes the classical case. 
The recent works \cite{gui2026forwardselfsimilarsolutionstwodimensional,lai2026forwardselfsimilarsolutionsnavierstokes} develop alternative constructions of self-similar solutions to the Navier--Stokes system in spatial dimensions beyond the classical three-dimensional setting.
An analogous approach was employed in \cite{BK25,tsai2024} to construct large self-similar solutions for Boussinesq-type equations.


Self-similar solutions also play a central role in the theory of nonlinear diffusion equations. Since the literature on this subject is extensive, we mention below only a few representative results that are particularly relevant to the present work.

The Barenblatt solution is the classical
self-similar solution of the porous medium equation
\[
u_t-\Delta u^m=0.
\]
Its role in the asymptotic behavior of general solutions is discussed
in \cite{Barenblatt,V03}. For equations with absorption, Kamin and
Peletier studied the heat equation with absorption and the porous
medium equation with absorption \cite{KP85,KP86}. Self-similar and
asymptotically self-similar solutions of the semilinear heat equation
\[
u_t-\Delta u-u^p=0
\]
have been studied, for example, in \cite{CW98,N12,Q18}. 
The work \cite{K99} provides a general scaling framework for the construction of self-similar solutions to nonlinear evolution equations arising from sufficiently small homogeneous initial data. Related fixed-point and asymptotic methods for the scalar convection--diffusion equation \eqref{eq} on the half-line were developed in \cite{BK00,K00}. In those works, ordinary differential equation techniques were used to construct non-integrable self-similar profiles.

More recently, self-similar solutions with both regular and singular profiles have been constructed for inviscid fluid equations, including the Euler and Boussinesq systems; see, for example, \cite{Abe24,Buckmaster23,Abe26}.
The generalized fractal Burgers equation, namely equation \eqref{eq}, where the Laplacian is replaced by the fractional Laplacian $(-\Delta)^{\alpha/2}$ with $\alpha \in (0,2)$, was considered by Jakubowski and Serafin \cite{JakubowskiSerafin2023}. They proved the existence of a mild solution to the equation supplemented with a large class of initial conditions. In particular, for $u_0(x)=C|x|^{-\beta}$, $C>0$, $1<\beta<N$, they constructed a self-similar solution of the Cauchy problem. In this way, they generalized previous results obtained in \cite{BKW01}.



\subsection{Notation}
We denote the norm in $L^p(\bbfR^N)$ by $\|\cdot\|_p$ and the
norm in a Banach space $\mathcal X$ by $\|\cdot\|_{\mathcal X}$.
We use the standard Sobolev spaces $W^{k,p}$.
We write $U_0=e^\Delta u_0$ for the heat evolution at time $t=1$
given by 
\begin{equation*}
  U_0(x)=(e^\Delta u_0)(x)=\int_{\bbfR^N} G(x-y,1)u_0(y)\ud y.
\end{equation*}
Here $G(x,t)=(4\pi t)^{-N/2}\exp\bigl(-|x|^2/(4t)\bigr)$ is the Gaussian heat kernel.
The required estimates are collected in the Appendix.
Constants in the estimates below are denoted by the same letter $C$, even when they vary from line to line. When useful, we indicate their dependence on the parameters in the calculation.

\section{Main result}\label{section-main-results}

Our main result gives self-similar solutions of equation~\eqref{eq}.

\begin{theorem}\label{main}
Assume that $N\geqslant1$ and
\begin{equation}\label{q}
1+\frac1N<q
\quad\text{if }N\in\{1,2\},
\qquad
1+\frac1N<q<\frac{N}{N-2}
\quad\text{if }N\geqslant3.
\end{equation}
Let
\begin{equation}\label{u0}
u_0(x)
=
\frac{C(x)}{|x|^{\frac1{q-1}}},
\qquad
C\in L^\infty(\bbfR^N),
\qquad
C(\lambda x)=C(x)
\end{equation}
for every $\lambda>0$ and a.e. $x\in\bbfR^N$.
Then the Cauchy problem
\begin{equation}\label{eq:main}
\begin{aligned}
u_t-\Delta u+\partial_{x_1}\bigl(|u|^{q-1}u\bigr)&=0,
&&x\in\bbfR^N,\quad t>0,\\
u(x,0)&=u_0(x),
&&x\in\bbfR^N,
\end{aligned}
\end{equation}
has a self-similar solution
\begin{equation}\label{u:self:main}
u(x,t)
=
t^{-\frac1{2(q-1)}}
U\left(\frac{x}{\sqrt t}\right),
\end{equation}
where the profile $U=U(x)$ satisfies, in a weak sense,  the elliptic equation
\begin{equation}\label{ell_U_q}
-\Delta U-\frac12x\cdot\nabla U
-\frac{1}{2(q-1)}U
+\partial_{x_1}\bigl(|U|^{q-1}U\bigr)=0
\qquad\text{in }\bbfR^N.
\end{equation}
Moreover, writing
\(
U_0=e^\Delta u_0,
\) 
the following properties hold:
\begin{itemize}
\item
if
\(
q>1+\frac2N,
\)
then
\(
U-U_0\in W^{1,2}(\bbfR^N);
\)

\item
if
\(
1+\frac1N<q\leqslant1+\frac2N,
\)
then
\(
U-U_0\in W^{1,p}(\bbfR^N)
\)
for every
\[
\max\left\{
1,\frac{N(q-1)}q
\right\}
<p<N(q-1).
\]
\end{itemize}
In addition,
\(U\in L^p(\bbfR^N)\) for every
\(p\in\bigl(N(q-1),\infty\bigr]\),
and the initial condition is attained in the sense that
\begin{equation}\label{initial:trace:main}
u(\cdot,t)\longrightarrow u_0
\qquad\text{in }L^r_{\rm loc}(\bbfR^N)
\quad\text{as }t\to0
\end{equation}
for every
\(
1\leqslant r<N(q-1).
\)
\end{theorem}





To prove the theorem, we truncate the nonlinearity $|U|^{q-1}U$ at a cutoff level $A > 0$ and consider the elliptic equation satisfied by $V = U - e^{\Delta} u_0$ (see equation~\eqref{ell_V:gen} below). In Theorem~\ref{existence_I}, we prove the existence of $V$ by combining \textit{a priori} estimates with the Leray--Schauder theorem, approximating solutions on $\mathbb{R}^N$ by solutions to Dirichlet boundary value problems on bounded domains. 
Finally, in Subsection~\ref{subsection:truncation:independent}, we establish $L^\infty$ estimates for solutions 
either to equation \eqref{ell_U:gen} or to equation \eqref{ell_V:gen} that are uniform with respect to $A$. We then choose $A$ larger than this uniform bound; the truncation is inactive on the range of the solution, so the resulting profile satisfies the original equation with nonlinearity $|U|^{q-1}U$.

\begin{remark} \label{rem:Lp:self}
Here, we record the natural integrability threshold for the self-similar profiles $U = U(x)$. The profiles constructed in Theorem~\ref{main} exhibit a structural barrier to fast spatial decay. Indeed, let $U$ be a bounded, sufficiently smooth solution to the profile equation~\eqref{ell_U_q} such that $U \in L^p(\mathbb{R}^N)$ for some $p > 1$. Testing equation \eqref{ell_U_q} against $|U|^{p-2}U$ and integrating over $\mathbb{R}^N$  (in the same way as in Lemma~\ref{lem:large:p:U}, below) yields
\begin{equation}\label{est:Lp:0}
\frac{4(p-1)}{p^2} \left\| \nabla |U|^{\frac{p}{2}} \right\|_2^2 + \left( \frac{N}{2p} - \frac{1}{2(q-1)} \right) \|U\|_p^p = 0,
\end{equation}
where the convection term vanishes identically upon integration by parts. Consequently, if
$ 
p \le N(q-1),
$ 
both terms in \eqref{est:Lp:0} are non-negative, forcing $U \equiv 0$. Thus, any nontrivial profile $U$ satisfying \eqref{est:Lp:0} must obey
\begin{equation*}
U \notin L^p(\mathbb{R}^N) \qquad \text{for all } 1 < p \le N(q-1).
\end{equation*}
In particular, when $q \ge 1 + \frac{2}{N}$, it follows that $U \notin L^2(\mathbb{R}^N)$. 

This cutoff matches the linear heat-profile estimate in Lemma~\ref{heat_L}, which yields $e^{\Delta} u_0 \in L^p(\mathbb{R}^N)$ for every $p>N(q-1)$.
This threshold aligns with the expected far-field asymptotics of $U(x)$. Because the initial datum $u_0$ is homogeneous of degree $-1/(q-1)$, one naturally expects $U$ to share the leading-order decay rate at spatial infinity:
$
U(x) = \mathcal{O}\!\left(|x|^{-{1}/(q-1)}\right)
$
as $ |x| \to \infty$.
\end{remark}

\begin{remark}
Let us explain the restriction
\(
q<{N}/({N-2})
\)
for 
$N\geqslant3$ in Theorem \ref{main}.
In the energy argument, \(V=U-U_0\) belongs to
\(W^{1,2}(\bbfR^N)\). To use
\(\varphi\in W^{1,2}(\bbfR^N)\) as a test function, we need
\(
\psi(U)\in L^2(\bbfR^N).
\)
Since \(|\psi(s)|\leqslant |s|^q\), it is sufficient to have
\(U=V+U_0\in L^{2q}(\bbfR^N)\).
The Sobolev embedding
\(
W^{1,2}(\bbfR^N)
\hookrightarrow L^{\frac{2N}{N-2}}(\bbfR^N)
\)
gives \(V\in L^{2q}(\bbfR^N)\) when
\(q\leqslant N/(N-2)\). On the other hand, Lemma \ref{heat_L}
implies that \(U_0\in L^{2q}(\bbfR^N)\) provided
\(
2q>N(q-1),
\)
or equivalently,
\(
q<{N}/({N-2}).
\)
Thus, the strict upper bound on \(q\) ensures that the nonlinear
term
\(
\int_{\bbfR^N}
\psi(U)\,\partial_{x_1}\varphi\,\ud x
\)
is well defined for  \(\vf \in W^{1,2}(\bbfR^N)\).
\end{remark}


\section{Equation with truncated nonlinearity}

\subsection{Elliptic equation with a globally Lipschitz nonlinearity}

We first consider the elliptic equation
\begin{equation}\label{ell_U:gen}
-\Delta U - \tfrac{1}{2}x\cdot \nabla U - \frac{1}{2(q-1)} U+\partial_{x_1}\psi (U)=0, \qquad x\in \bbfR^N,
\end{equation}
with the truncated nonlinearity
\begin{align}\label{psi1}
\psi(V)=\begin{cases}
A^q, \quad &\text{if}\quad V>A,\\
|V|^{q-1}V, \quad & \text{if}\quad |V|\leqslant A,\\
-A^q, \quad &\text{if}\quad V<-A
\end{cases}
\end{align}
where $q>1$ and $A>0$. Then
\begin{equation}\label{psi2}
\psi \in W^{1,\infty}(\bbfR) \quad \text{and}\quad \psi'(V)=0\quad \text{if}\quad |V|> A
\end{equation}
and
\begin{equation}\label{psi_estimate}
\| \psi\|_{L^\infty(\bbfR)}\leqslant A^q \quad \text{and} \quad \| \psi'\|_{L^\infty (\bbfR)}\leqslant qA^{q-1}.
\end{equation}
Moreover, a direct calculation shows that there is a constant $C(q)>0$, independent of $A$, such that 
\begin{equation}\label{psi:ineq}
|\psi(a+b)-\psi(a)|
\leqslant C(q)|b|\bigl(|a|^{q-1}+|b|^{q-1}\bigr)
\qquad\text{for all }a,b\in\bbfR.
\end{equation}

For the homogeneous initial datum \eqref{u0}, we seek a solution of
\eqref{ell_U:gen} in the form
\[
U=V+U_0,
\qquad U_0:=e^\Delta u_0.
\]
Since $U_0=e^\Delta u_0$ satisfies the linear elliptic equation \eqref{eq:U0:A} below, the function $V=V(x)$ solves
\begin{equation}\label{ell_V:gen}
-\Delta V-\frac{x}{2}\cdot\nabla V-\frac{1}{2(q-1)}V
+\partial_{x_1}\psi(V+U_0)=0
\qquad\text{in }\bbfR^N.
\end{equation}

We call $V$ a weak solution of \eqref{ell_V:gen} if
\begin{equation}\label{vf_V:gen}
\begin{split}
\int_{\bbfR^N}\nabla V\cdot\nabla\vf\ud x
-\int_{\bbfR^N}\frac{x}{2}\cdot\nabla V\,\vf\ud x
-\frac{1}{2(q-1)}&
\int_{\bbfR^N}V\vf\ud x\\
&-
\int_{\bbfR^N}
\psi(V+e^\Delta u_0)\partial_{x_1}\vf\ud x=0
\end{split}
\end{equation}
for every $\vf\in C^\infty_c(\bbfR^N)$.

\begin{theorem}\label{existence_I}
Assume $N\geqslant1$ and
\[
1+\frac1N<q
\quad\text{if }N\in\{1,2\},
\qquad
1+\frac1N<q<\frac{N}{N-2}
\quad\text{if }N\geqslant3.
\]
Let $u_0$ be of the form \eqref{u0}, and let $\psi$ be given by
formula \eqref{psi1} with an arbitrary $A>0$. Then equation
\eqref{ell_V:gen} has a weak solution $V=V(x)$ with the following
properties:
\begin{itemize}
\item
if
\(
q>1+\frac2N,
\)
then
\(
V\in W^{1,2}(\bbfR^N);
\)
\item
if
\(
q\in\left(1+\frac1N,1+\frac2N\right],
\)
then
\[
V\in W^{1,p}(\bbfR^N)
\qquad\text{for every}\qquad
\max\left\{1,\frac{N(q-1)}{q}\right\}
<p<N(q-1).
\]
\end{itemize}
\end{theorem}

The proof of Theorem \ref{existence_I} is divided into several steps, which are presented in the following subsections.

\subsection{Construction of solutions on bounded domains}
First, we construct a solution of equation \eqref{ell_V:gen} in a bounded domain $\Omega$ with smooth boundary, subject to the Dirichlet boundary condition. 
We treat the following ranges of the nonlinear exponent separately:
\begin{equation*}
q\geqslant 1+\frac{2}{N}\quad\text{(Lemma~\ref{C_omega:1})}
\qquad\text{and}\qquad
q\in\left(1+\frac{1}{N},1+\frac{2}{N}\right)
\quad\text{(Lemma~\ref{lemma_2})}.
\end{equation*}

\begin{lemma}\label{C_omega:1} Let $N\geq 1$.
Let $q\geqslant 1+\frac{2}{N}$, let $\Omega\subset\bbfR^N$ be a bounded
domain with $C^2$ boundary,
let $u_0$ be defined by \eqref{u0} and 
let $U_0=e^\Delta u_0$, and let
\(\psi\in W^{1,\infty}(\bbfR)\) be arbitrary.
Then the boundary value problem
\begin{align}\label{ell:bounded_4}
-\Delta V-\frac{x}{2}\cdot\nabla V-\frac{1}{2(q-1)}V
+\partial_{x_1}\psi(V+U_0)&=0,
&&x\in\Omega,\\
V&=0,
&&x\in\partial\Omega,
\end{align}
has a weak solution $V\in W^{1,2}_0(\Omega)$.
\end{lemma}

\begin{proof}
Set
\[
c=\frac{1}{2(q-1)}.
\]
For \(W\in W^{1,2}_0(\Omega)\), define \(T(W)=Z\), where \(Z\in
W^{1,2}_0(\Omega)\) is the unique weak solution of
\[
-\Delta Z
=
\frac{x}{2}\cdot\nabla W+cW
-\partial_{x_1}\psi(W+U_0),
\qquad
Z|_{\partial\Omega}=0.
\]
Since \(\psi\in L^\infty(\bbfR)\), the right-hand side belongs to
\(W^{-1,2}(\Omega)\), and hence \(T\) is well defined by the
Lax--Milgram theorem.

The map
\(
T:W^{1,2}_0(\Omega)\longrightarrow W^{1,2}_0(\Omega)
\)
is continuous and compact. Indeed, continuity follows from the Lipschitz
continuity of \(\psi\) and the continuity of the inverse Dirichlet
Laplacian. To prove compactness, write \(T=T_1+T_2\), where
\[
T_1(W)
=
(-\Delta_D)^{-1}
\left(\frac{x}{2}\cdot\nabla W+cW\right),
\qquad
T_2(W)
=
(-\Delta_D)^{-1}
\left(-\partial_{x_1}\psi(W+U_0)\right).
\]
The map \(T_1\) is compact from \(W^{1,2}_0(\Omega)\) into itself by
elliptic \(W^{2,2}\)-regularity and the compact embedding
\(W^{2,2}(\Omega)\hookrightarrow W^{1,2}(\Omega)\). Moreover, if
\(\{W_n\}\) is bounded in \(W^{1,2}_0(\Omega)\), then, after passing to a
subsequence,
\(
W_n\to W
\)
in $L^2(\Omega)$.
Since \(\psi\) is globally Lipschitz,
\[
\psi(W_n+U_0)\to\psi(W+U_0)
\qquad\text{in }L^2(\Omega),
\]
and therefore \(T_2(W_n)\to T_2(W)\) in \(W^{1,2}_0(\Omega)\). Thus \(T\)
is compact.

It remains to obtain the Leray--Schauder {\it a priori} bound. Let
\(
V=\lambda T(V)
\)
for $\lambda \in [0,1]$.
Then
\[
-\Delta V
=
\lambda\left(
\frac{x}{2}\cdot\nabla V+cV
-\partial_{x_1}\psi(V+U_0)
\right).
\]
This equation holds in $\Omega$, together with the homogeneous Dirichlet
boundary condition.
Testing this equation with \(V\) and using the relation
\[
\frac12\int_\Omega (x\cdot\nabla V)V\,\ud x
=
-\frac N4\int_\Omega V^2\,\ud x,
\]
we obtain
\[
\|\nabla V\|_2^2
+
\lambda\left(
\frac N4-\frac{1}{2(q-1)}
\right)\|V\|_2^2
=
\lambda\int_\Omega
\psi(V+U_0)\partial_{x_1}V\,\ud x.
\]
Since \(q\geq 1+2/N\), the second term on the left-hand side is
nonnegative. Hence
\[
\|\nabla V\|_2^2
\leqslant
|\Omega|^{1/2}\|\psi\|_\infty\|\nabla V\|_2,
\]
and consequently
\(
\|\nabla V\|_2
\leqslant
|\Omega|^{1/2}\|\psi\|_\infty.
\)
By the Poincar\'e inequality, this gives a bound for \(V\) in
\(W^{1,2}_0(\Omega)\), uniform with respect to \(\lambda\in[0,1]\).
The Leray--Schauder theorem therefore yields a fixed point
\(V\in W^{1,2}_0(\Omega)\) of \(T\), which is a weak solution of
\eqref{ell:bounded_4}.
\end{proof}

For
$q\in (1+1/N,1+2/N)$,
we construct a solution of the boundary value problem \eqref{ell:bounded_4}
in $W^{1,p}_0(\Omega)$, where
\(
1<p<N(q-1)<2.
\)
We begin with the following elementary lemma.

\begin{lemma}\label{lemma_1}
Let $1<p<2$ and let $\Omega\subset\bbfR^N$ be an arbitrary domain.
Suppose that
\[
|V|^{\frac p2-1}V\in W^{1,2}_0(\Omega).
\]
Then $V\in W^{1,p}_0(\Omega)$ and
\begin{equation}\label{lemma_1:estimate}
\int_\Omega |\nabla V|^p\ud x
\leqslant
C(p)
\left(
\int_\Omega
\left|\nabla\big(|V|^{\frac p2-1}V\big)\right|^2\ud x
+
\int_\Omega |V|^p\ud x
\right).
\end{equation}
If, in addition, $\Omega$ is bounded, then
\begin{equation}\label{lemma_1:estimate:bounded}
\int_\Omega |\nabla V|^p\ud x
\leqslant
C(p,\Omega)
\int_\Omega
\left|\nabla\big(|V|^{\frac p2-1}V\big)\right|^2\ud x.
\end{equation}
\end{lemma}

\begin{proof}
Set
\[
w=|V|^{\frac p2-1}V,
\quad 
\text{then}
\quad 
V=|w|^{\frac2p-1}w.
\]
For smooth $w$ we have
\[
|\nabla V|
=
\frac2p |w|^{\frac2p-1}|\nabla w|,
\]
and consequently, by H\"older's inequality,
\begin{align*}
\int_\Omega|\nabla V|^p\ud x
&=
\left(\frac2p\right)^p
\int_\Omega
|w|^{2-p}|\nabla w|^p\ud x
\\
&\leqslant
\left(\frac2p\right)^p
\left(\int_\Omega|\nabla w|^2\ud x\right)^{p/2}
\left(\int_\Omega|w|^2\ud x\right)^{(2-p)/2}
\\
&\leqslant
C(p)
\left(
\int_\Omega|\nabla w|^2\ud x
+
\int_\Omega|w|^2\ud x
\right).
\end{align*}
The general case follows by approximation in $W^{1,2}_0(\Omega)$, which we describe in detail in Lemma \ref{C_psi_1} below.
Since $|w|^2=|V|^p$, this proves inequality \eqref{lemma_1:estimate}.
If $\Omega$ is bounded, inequality \eqref{lemma_1:estimate:bounded} follows
from the Poincar\'e inequality applied to $w$.
\end{proof}

\begin{lemma}\label{lemma_2}
Let
\(
q\in\left(1+\frac1N,1+\frac2N\right]
\)
and let $\Omega\subset\bbfR^N$ be a bounded domain with $C^2$ boundary.
Let $\psi\in W^{1,\infty}(\bbfR)$ be given by formula \eqref{psi1}  
with some $A>0$. 
Let $u_0$ be defined by \eqref{u0} and set $U_0=e^\Delta u_0$.
For every
\(
1<p<N(q-1),
\)
the boundary value problem
\begin{equation}\label{ell:bounded}
\begin{aligned}
-\Delta V-\frac{x}{2}\cdot\nabla V
-\frac{1}{2(q-1)}V
&=
-\partial_{x_1}\psi (V+e^\Delta u_0),
&&x\in\Omega,\\
V&=0,
&&x\in\partial\Omega,
\end{aligned}
\end{equation}
has a weak solution $V\in W^{1,p}_0(\Omega)$.
\end{lemma}

\begin{proof}
Set
\[
c=\frac{1}{2(q-1)}\quad \text{and} \qquad X=W^{1,p}_0(\Omega).
\]
Since \(q\leq 1+2/N\) and \(p<N(q-1)\), we have \(1<p<2\).
Moreover, by Lemma \ref{heat_L}, we have 
\(
U_0,\nabla U_0\in L^\infty(\bbfR^N).
\)

For \(W\in X\), let \(T(W)=Z\), where \(Z\in X\) solves
\[
-\Delta Z
=
\frac{x}{2}\cdot\nabla W+cW
-\partial_{x_1}\psi(W+U_0),
\qquad
Z|_{\partial\Omega}=0.
\]
Since \(\psi\) is bounded and globally Lipschitz, standard elliptic
regularity estimates show that \(T:X\to X\) is well defined and continuous.
Moreover, the proof of Lemma~\ref{C_omega:1} shows that \(T\) is compact.

We apply the Leray--Schauder theorem. Let
\(
V=\lambda T(V)
\)
for $\lambda \in [0,1]$.
Then
\begin{equation}\label{ell:bounded:homotopy:p}
-\Delta V
=
\lambda\left(
\frac{x}{2}\cdot\nabla V+cV
-\partial_{x_1}\psi(V+U_0)
\right).
\end{equation}
Standard elliptic bootstrap, using the chain rule in the nonlinear term and the boundedness of \(\psi'\), yields
enough regularity to justify the calculations below.
Using a standard regularization of \(|V|^{p-2}V\) as a test function in
equation
\eqref{ell:bounded:homotopy:p} 
and then passing to the limit
(see the proof of Lemma \ref{C_psi_1} below for more detailed reasoning), we obtain
\begin{equation}\label{p_est_correct}
\begin{split}
&\frac{4(p-1)}{p^2}
\int_\Omega
\left|
\nabla\bigl(|V|^{\frac p2-1}V\bigr)
\right|^2\ud x
+
\lambda\left(
\frac{N}{2p}-\frac{1}{2(q-1)}
\right)
\int_\Omega |V|^p\ud x
\\
&\qquad\leqslant
\lambda
\left|
\int_\Omega
\psi'(V+U_0)
\bigl(\partial_{x_1}V+\partial_{x_1}U_0\bigr)
|V|^{p-2}V\,\ud x
\right|.
\end{split}
\end{equation}
Here we used
\[
\frac12\int_\Omega
(x\cdot\nabla V)|V|^{p-2}V\,\ud x
=
-\frac{N}{2p}\int_\Omega |V|^p\,\ud x.
\]
Since \(p<N(q-1)\), the second term on the left-hand side of
\eqref{p_est_correct} is nonnegative.

By the definition \eqref{psi1} of the truncated nonlinearity,
\[
\psi'(V+U_0)=0
\qquad\text{a.e. on }\{|V+U_0|>A\}.
\]
Consequently, on the support of \(\psi'(V+U_0)\), we have
\(
|V|\leq A+\|U_0\|_\infty,
\)
and therefore
\[
\left|
\int_\Omega
\psi'(V+U_0)
\bigl(\partial_{x_1}V+\partial_{x_1}U_0\bigr)
|V|^{p-2}V\,\ud x
\right|
\leq
C\left(
\|\nabla V\|_{L^1(\Omega)}+1
\right),
\]
where   \(C=C(U_0,|\Omega|,A)>0\) is independent of \(\lambda\).

Set
\[
E=
\int_\Omega
\left|
\nabla\bigl(|V|^{\frac p2-1}V\bigr)
\right|^2\ud x.
\]
By Lemma \ref{lemma_1},
\[
\|\nabla V\|_p^p\leq C E,
\]
and hence
\[
\|\nabla V\|_1
\leq
C\|\nabla V\|_p
\leq C E^{1/p}.
\]
Dropping the nonnegative \(L^p\)-term in \eqref{p_est_correct}, we obtain
\[
E\leq C E^{1/p}+C.
\]
Since \(p>1\), the $\varepsilon$-Young inequality gives
\(
E\leq C,
\)
uniformly with respect to \(\lambda\in[0,1]\). Lemma \ref{lemma_1} and
the Poincar\'e inequality then yield
\(
\|V\|_{W^{1,p}_0(\Omega)}
\leq C,
\)
again independently of \(\lambda\).
Thus all the assumptions of the Leray--Schauder theorem are satisfied.
Hence \(T\) has a fixed point \(V\in W^{1,p}_0(\Omega)\), which is a weak
solution of \eqref{ell:bounded}.
\end{proof}

\subsection{Domain-independent estimates on expanding balls}

A weak solution of equation \eqref{ell_V:gen} in the whole space is obtained as a limit of solutions 
$\{V_k\}_{k\in \bbfN} $ to the Dirichlet boundary value problems for each $k\in\bbfN$:
\begin{equation}\label{Vk_equ}
\begin{aligned}
-\Delta V_k-\frac{x}{2}\cdot\nabla V_k
-\frac{1}{2(q-1)}V_k
+
\partial_{x_1}\psi(V_k+U_0)&=0,
&&x\in B_k,\\
V_k&=0,
&&x\in\partial B_k,
\end{aligned}
\end{equation}
where  \(
B_k=\{x\in\bbfR^N:\ |x|<k\}
\) is 
an increasing sequence of balls. 

In this subsection, we derive \(W^{1,p}\) estimates for these approximations that are uniform with respect to the balls $B_k$. We treat the following cases separately:
\[
q>1+\frac{2}{N}
\qquad\text{and}\qquad
q\in\left(1+\frac{1}{N},1+\frac{2}{N}\right].
\]

\begin{lemma}\label{C_psi}
Assume that
\[
q>1+\frac{2}{N}
\quad\text{if }N\in\{1,2\},
\qquad
1+\frac{2}{N}<q<\frac{N}{N-2}
\quad\text{if }N\geqslant3.
\]
Let $\psi=\psi_A$ be the truncated nonlinearity defined in
\eqref{psi1}, with an arbitrary truncation level $A>0$. Let
$u_0$ be given by \eqref{u0} and set
\(
U_0=e^\Delta u_0.
\)
Then, for every $k\in\mathbb N$, the boundary value problem
\eqref{Vk_equ} has a weak solution
\(
V_k\in W^{1,2}_0(B_k).
\)
Moreover, there exists a constant
\(
C=C(A,N,q,U_0)>0,
\)
independent of $k$ and of the chosen solution $V_k$, such that
\begin{equation}\label{Vk_psi}
\|\nabla V_k\|_{L^2(B_k)}^2
+
\|V_k\|_{L^2(B_k)}^2
\leqslant C.
\end{equation}

If, in addition,
$ 
1+\frac{2}{N}<q<2,
$ 
then the constant in inequality  \eqref{Vk_psi} can be chosen in the form
\(
C=C(N,q,U_0),
\)
and is therefore independent of both $k$ and the truncation level
$A$.
\end{lemma}

\begin{proof}
The existence of $V_k$ follows from Lemma \ref{C_omega:1}.
It remains to prove the stated estimates.
Set
\[
\alpha
=
\frac{N}{4}-\frac{1}{2(q-1)}.
\]
The assumption $q>1+2/N$ implies that
$ 
\alpha>0.
$ 

Testing equation \eqref{Vk_equ} with $V_k$ and using
$V_k=0$ on $\partial B_k$, we obtain
\begin{align}
\|\nabla V_k\|_{L^2(B_k)}^2
+
\alpha\|V_k\|_{L^2(B_k)}^2
&=
\int_{B_k}
\psi_A(V_k+U_0)\,
\partial_{x_1}V_k\,dx.
\label{C_psi:energy}
\end{align}

Let
\(
\Psi_A(s)=\int_0^s\psi_A(z)\,dz.
\)
Since $\psi_A$ is globally Lipschitz, the Sobolev chain rule gives
\[
\partial_{x_1}\Psi_A(V_k)
=
\psi_A(V_k)\partial_{x_1}V_k
\qquad\text{a.e. in }B_k.
\]
Because $V_k=0$ on $\partial B_k$ and $\Psi_A(0)=0$, integration
by parts yields
\[
\int_{B_k}
\psi_A(V_k)\partial_{x_1}V_k\,dx
=
\int_{B_k}\partial_{x_1}\Psi_A(V_k)\,dx
=
0.
\]
Consequently,
\begin{align}
\left|
\int_{B_k}
\psi_A(V_k+U_0)\partial_{x_1}V_k\,dx
\right|
&=
\left|
\int_{B_k}
\bigl(
\psi_A(V_k+U_0)-\psi_A(V_k)
\bigr)
\partial_{x_1}V_k\,dx
\right|.
\label{C_psi:difference}
\end{align}

The truncations $\psi_A$ satisfy
\begin{equation}\label{C_psi:uniform-difference}
|\psi_A(a+b)-\psi_A(a)|
\leqslant
C(q)|b|
\bigl(
|a|^{q-1}+|b|^{q-1}
\bigr),
\end{equation}
where $C(q)$ is independent of $A$. Applying
\eqref{C_psi:uniform-difference} with $a=V_k$ and $b=U_0$ gives
\begin{align}
&\left|
\int_{B_k}
\bigl(
\psi_A(V_k+U_0)-\psi_A(V_k)
\bigr)
\partial_{x_1}V_k\,dx
\right|
\nonumber\\
&\qquad\qquad
\leqslant
C(q)\int_{B_k}
\left(
|U_0|^q
+
|U_0||V_k|^{q-1}
\right)
|\nabla V_k|\,dx.
\label{C_psi:basic-estimate}
\end{align}

We first consider the subquadratic case
$1+\frac{2}{N}<q<2$.
For every $\varepsilon>0$, the $\varepsilon$-Young inequality gives
\begin{align}
\int_{B_k}|U_0|^q|\nabla V_k|\,dx
&\leqslant
\varepsilon\|\nabla V_k\|_{L^2(B_k)}^2
+
C(\varepsilon)
\|U_0\|_{L^{2q}(\mathbb R^N)}^{2q}.
\label{C_psi:U0q}
\end{align}

For the second term, use the $\varepsilon$-Young inequality with the three
conjugate exponents
$
{2}/({2-q}),
$ 
$
{2}/({q-1}),
$ and 
$
2.
$
After inserting suitable $\varepsilon$-factors, we obtain
\begin{align}
|U_0||V_k|^{q-1}|\nabla V_k|
&\leqslant
C(\varepsilon,q)
|U_0|^{\frac{2}{2-q}}
+
\varepsilon |V_k|^2
+
\varepsilon|\nabla V_k|^2.
\end{align}
Therefore,
\begin{align}
\int_{B_k}
|U_0||V_k|^{q-1}|\nabla V_k|\,dx
&\leqslant
\varepsilon
\left(
\|V_k\|_{L^2(B_k)}^2
+
\|\nabla V_k\|_{L^2(B_k)}^2
\right)
\nonumber\\
&\quad
+
C(\varepsilon,q)
\|U_0\|_{
L^{\frac{2}{2-q}}(\mathbb R^N)
}^{\frac{2}{2-q}}.
\label{C_psi:mixed-subquadratic}
\end{align}

The two norms of $U_0$ appearing above are finite. Indeed, the
heat-profile estimate in Lemma \ref{heat_L} gives
\(
U_0\in L^r(\mathbb R^N)
\)
whenever
\(
r>N(q-1).
\)
First, the inequality
\(
2q>N(q-1)
\)
is equivalent, for $N\geqslant3$, to
\(
q<\frac{N}{N-2}.
\)
Second, to verify the inequality 
\(
\frac{2}{2-q}>N(q-1),
\)
observe that
\(
q-1<\frac{2}{N-2}
\)
by the upper restriction on $q$, while
\(
2-q<\frac{N-2}{N}
\)
because $q>1+2/N$. Hence
\[
(q-1)(2-q)
<
\frac{2}{N-2}\frac{N-2}{N}
=
\frac2N,
\]
which is equivalent to
\(
\frac{2}{2-q}>N(q-1).
\)

Combining \eqref{C_psi:energy},
\eqref{C_psi:difference}, \eqref{C_psi:basic-estimate},
\eqref{C_psi:U0q}, and
\eqref{C_psi:mixed-subquadratic}, we obtain
\begin{align}
\|\nabla V_k\|_2^2+\alpha\|V_k\|_2^2
&\leqslant
C(q)\varepsilon
\left(
\|\nabla V_k\|_2^2+\|V_k\|_2^2
\right)
\nonumber\\
&\quad
+
C(\varepsilon,q)
\left(
\|U_0\|_{2q}^{2q}
+
\|U_0\|_{\frac{2}{2-q}}^{\frac{2}{2-q}}
\right).
\label{C_psi:subquadratic-final}
\end{align}
Choosing $\varepsilon>0$ sufficiently small
and absorbing the first term on the right-hand side into the
left-hand side, we obtain
\[
\|\nabla V_k\|_{L^2(B_k)}^2
+
\|V_k\|_{L^2(B_k)}^2
\leqslant
C(N,q,U_0).
\]
Every constant used in this part is independent of $A$.
This proves the asserted truncation-independent estimate when
$1+2/N<q<2$.

It remains to discuss the case $q\geqslant2$. For this range the
argument gives an estimate for each fixed $A$, but not an estimate
uniform in $A$.
Set
\[
M=A+\|U_0\|_{L^\infty(\mathbb R^N)}
\]
and
\[
\Omega_{k,A}
=
\{x\in B_k:\ |V_k(x)|\leqslant M\}.
\]
If $x\in B_k\setminus\Omega_{k,A}$, then $V_k(x)$ and
$V_k(x)+U_0(x)$ have the same sign and both have absolute value
larger than $A$. Hence
\[
\psi_A(V_k+U_0)=\psi_A(V_k)
\qquad\text{on }B_k\setminus\Omega_{k,A}.
\]
Thus the integral in \eqref{C_psi:difference} is supported in
$\Omega_{k,A}$.

Fix $R>0$. On $\Omega_{k,A}\cap B_R$, the estimate
$\|\psi_A\|_\infty\leqslant A^q$ gives
\begin{align}
&\left|
\int_{\Omega_{k,A}\cap B_R}
\bigl(
\psi_A(V_k+U_0)-\psi_A(V_k)
\bigr)
\partial_{x_1}V_k\,dx
\right|
\nonumber\\
&\qquad
\leqslant
2A^q|B_R|^{1/2}
\|\nabla V_k\|_{L^2(B_k)}
\nonumber\\
&\qquad
\leqslant
\varepsilon\|\nabla V_k\|_{L^2(B_k)}^2
+
C(\varepsilon)A^{2q}|B_R|.
\label{C_psi:local-A}
\end{align}
This term already exhibits explicit dependence on $A$.

On $\Omega_{k,A}$ we have $|V_k|\leqslant M$, and hence, because
$q\geqslant2$, we obtain
\[
|V_k|^{q-1}
\leqslant
M^{q-2}|V_k|.
\]
It follows that
\begin{align}
&\int_{\Omega_{k,A}\cap B_R^c}
|U_0||V_k|^{q-1}|\nabla V_k|\,dx
\nonumber\\
&\qquad
\leqslant
\frac{M^{q-2}}{2}
\|U_0\|_{L^\infty(B_R^c)}
\left(
\|V_k\|_{L^2(B_k)}^2
+
\|\nabla V_k\|_{L^2(B_k)}^2
\right).
\label{C_psi:exterior-A}
\end{align}
Since
\(
\|U_0\|_{L^\infty(B_R^c)}
\to 0
\)
as $R\to \infty$,
we may choose $R$ sufficiently large that the coefficient in
\eqref{C_psi:exterior-A} can be absorbed into the left-hand side
of \eqref{C_psi:energy}. In general, this choice depends on
\[
M^{q-2}
=
\bigl(
A+\|U_0\|_\infty
\bigr)^{q-2},
\]
and therefore depends on $A$ when $q>2$.

The remaining exterior term is estimated by
\[
\int_{B_k}|U_0|^q|\nabla V_k|\,dx
\leqslant
\varepsilon\|\nabla V_k\|_2^2
+
C(\varepsilon)\|U_0\|_{2q}^{2q}.
\]
Choosing first $R$ and then $\varepsilon$ sufficiently small and
absorbing the energy terms, we obtain
\[
\|\nabla V_k\|_{L^2(B_k)}^2
+
\|V_k\|_{L^2(B_k)}^2
\leqslant
C(A,N,q,U_0).
\]
The constant is independent of $k$ and $V_k$, but the argument
does not make it independent of $A$. The dependence on $A$ occurs
explicitly in \eqref{C_psi:local-A} and, for $q>2$, also in the
choice of $R$ following inequality \eqref{C_psi:exterior-A}.
\end{proof}

\begin{lemma}\label{C_psi_1}
Let $N\geqslant1$ and assume that
\(
q\in\left(1+\frac1N,\,1+\frac2N\right].
\)
Let
\[
\max\left\{
1,\frac{N(q-1)}q
\right\}
<p<N(q-1),
\]
and let $\psi=\psi_A$ be the truncated nonlinearity defined by \eqref{psi1} with a truncation level $A>0$.
Let $u_0$ be defined by \eqref{u0}, set
\(
U_0=e^\Delta u_0,
\)
and suppose that
\(
V_k\in W^{1,p}_0(B_k)
\)
is a weak solution of
\begin{equation}\label{ell:bounded:uni:p:Cpsi1}
\begin{aligned}
-\Delta V_k-\frac{x}{2}\cdot\nabla V_k
-\frac{1}{2(q-1)}V_k
+\partial_{x_1}\psi_A(V_k+U_0)&=0
&&\text{in }B_k,\\
V_k&=0
&&\text{on }\partial B_k.
\end{aligned}
\end{equation}
Then
\begin{equation}\label{Cpsi1:estimate}
\|V_k\|_{W^{1,p}_0(B_k)}
\leqslant
C(A,N,q,p,U_0),
\end{equation}
where the constant is independent of $k$ and of the particular
solution $V_k$.

Moreover, if
\(
N\geqslant2
\)
and
\(
q<2,
\)
then the constant in \eqref{Cpsi1:estimate} can be chosen
independently of the truncation level $A$. 
\end{lemma}

\begin{proof}
For simplicity, write
\[
V=V_k
\qquad\text{and}\qquad 
c=\frac{1}{2(q-1)}.
\]
All constants below are independent of $k$.

In the following reasoning, we present a regularization procedure which allows us to obtain $W^{1,p}$-estimates for $p\in (1,2)$ in this proof as well as in other parts of this paper.
For $\delta>0$, define
\[
\Phi_\delta(s)
=
s(s^2+\delta)^{\frac{p-2}{2}},
\qquad
Q_\delta(s)
=
\int_0^s\Phi_\delta(\tau)\,\ud\tau.
\]
Then $\Phi_\delta$ is increasing and
\begin{equation}\label{Cpsi1:Phi-properties}
0\leqslant\Phi_\delta'(s)
\leqslant C(p)|s|^{p-2},
\qquad s\neq0.
\end{equation}
Moreover,
\[
\Phi_\delta(s)\longrightarrow |s|^{p-2}s,
\qquad
Q_\delta(s)\longrightarrow \frac1p|s|^p
\]
as $\delta\downarrow0$.
Testing the boundary value problem \eqref{ell:bounded:uni:p:Cpsi1} with
$\Phi_\delta(V)$ gives
\begin{align}
&\int_{B_k}\Phi_\delta'(V)|\nabla V|^2\,dx
+\frac N2\int_{B_k}Q_\delta(V)\,dx
-c\int_{B_k}V\Phi_\delta(V)\,dx
\nonumber\\
&\qquad
=
\int_{B_k}
\psi_A(V+U_0)\Phi_\delta'(V)\,
\partial_{x_1}V\,dx.
\label{Cpsi1:regularized-energy}
\end{align}

Introduce
\[
P_\delta(v,u)
=
\int_0^v
\psi_A(s+u)\Phi_\delta'(s)\,ds.
\]
Then
\[
\partial_vP_\delta(v,u)
=
\psi_A(v+u)\Phi_\delta'(v).
\]
Consequently,
\[
\partial_{x_1}P_\delta(V,U_0)
=
\psi_A(V+U_0)\Phi_\delta'(V)\partial_{x_1}V
+
\partial_uP_\delta(V,U_0)\partial_{x_1}U_0.
\]
Because \(V=0\) on \(\partial B_k\) and
\(P_\delta(0,u)=0\), integration over \(B_k\) gives
\begin{equation}\label{Cpsi1:P-integration}
\int_{B_k}
\psi_A(V+U_0)\Phi_\delta'(V)\partial_{x_1}V\,dx
=
-\int_{B_k}
\partial_uP_\delta(V,U_0)\partial_{x_1}U_0\,dx.
\end{equation}
Combining \eqref{Cpsi1:regularized-energy} and
\eqref{Cpsi1:P-integration}, we obtain
\begin{align}
&\int_{B_k}\Phi_\delta'(V)|\nabla V|^2\,dx
+\frac N2\int_{B_k}Q_\delta(V)\,dx
-c\int_{B_k}V\Phi_\delta(V)\,dx
\nonumber\\
&\qquad
=
-\int_{B_k}
\partial_uP_\delta(V,U_0)\partial_{x_1}U_0\,dx.
\label{Cpsi1:regularized-energy:P}
\end{align}

 We next estimate the last integral using
the truncation-independent pointwise estimate
$ 
|\psi_A'(z)|
\leqslant q|z|^{q-1}.
$ 
Therefore, by \eqref{Cpsi1:Phi-properties},
\begin{align}
|\partial_uP_\delta(v,u)|
&=
\left|
\int_0^v
\psi_A'(s+u)\Phi_\delta'(s)\,ds
\right|
\nonumber\\
&\leqslant
C(p,q)
\int_0^{|v|}
(s+|u|)^{q-1}s^{p-2}\,ds
\nonumber\\
&\leqslant
C(p,q)
\left(
|v|^{p+q-2}
+
|u|^{q-1}|v|^{p-1}
\right).
\label{Cpsi1:Pu-uniform}
\end{align}
The constant in \eqref{Cpsi1:Pu-uniform} is independent of
both \(A\) and \(\delta\).

We distinguish the relevant cases.

\medskip
\noindent
\textit{Case 1: \(N\geqslant2\) and \(q<2\).}
Using inequality \eqref{Cpsi1:Pu-uniform}, we obtain
\begin{align}
\left|
\int_{B_k}
\partial_uP_\delta(V,U_0)\partial_{x_1}U_0\,dx
\right|
&\leqslant
C
\int_{B_k}
|\nabla U_0||V|^{p+q-2}\,dx
\nonumber\\
&\quad
+
C
\int_{B_k}
|U_0|^{q-1}|\nabla U_0||V|^{p-1}\,dx.
\label{Cpsi1:RHS-two-terms}
\end{align}
The Young inequality gives, for every \(\varepsilon>0\),
\begin{align}
\int_{B_k}
|\nabla U_0||V|^{p+q-2}\,dx
&\leqslant
\varepsilon\|V\|_{L^p(B_k)}^p
+
C_\varepsilon
\|\nabla U_0\|_
 {L^{p/(2-q)}(\mathbb R^N)}^{p/(2-q)},
\label{Cpsi1:first-young}\\
\int_{B_k}
|U_0|^{q-1}|\nabla U_0||V|^{p-1}\,dx
&\leqslant
\varepsilon\|V\|_{L^p(B_k)}^p
+
C_\varepsilon
\bigl\|
 |U_0|^{q-1}\nabla U_0
\bigr\|_{L^p(\mathbb R^N)}^p.
\label{Cpsi1:second-young}
\end{align}
The assumptions on \(p\) imply
\[
\frac{p}{2-q}>N(q-1)
\qquad\text{and}\qquad
pq>N(q-1).
\]
Indeed,
\[
p>\frac{N(q-1)}q,
\qquad
q>1
\quad\Longrightarrow\quad
2-q<\frac1q.
\]
Thus, Lemma \ref{heat_L}
implies 
\[
\nabla U_0\in
L^{p/(2-q)}(\mathbb R^N)
\quad \text{and}\quad 
U_0,\nabla U_0\in L^{pq}(\mathbb R^N).
\]
In particular, the H\"older inequality yields
\[
\bigl\|
 |U_0|^{q-1}\nabla U_0
\bigr\|_{L^p}
\leqslant
\|U_0\|_{L^{pq}}^{q-1}
\|\nabla U_0\|_{L^{pq}}.
\]
It follows from \eqref{Cpsi1:RHS-two-terms}--%
\eqref{Cpsi1:second-young} that
\begin{equation}\label{Cpsi1:RHS-uniform}
\left|
\int_{B_k}
\partial_uP_\delta(V,U_0)\partial_{x_1}U_0\,dx
\right|
\leqslant
2\varepsilon\|V\|_{L^p(B_k)}^p
+
C_\varepsilon(N,q,p,U_0),
\end{equation}
where the constant is independent of \(A\), \(k\), and
\(\delta\).

Letting \(\delta\downarrow0\) in
\eqref{Cpsi1:regularized-energy:P}, using Fatou's lemma for the
gradient term and the standard convergence of the regularized
powers, gives
\begin{align}
\frac{4(p-1)}{p^2}
\left\|
 \nabla\bigl(|V|^{\frac p2-1}V\bigr)
\right\|_{L^2(B_k)}^2
+
\left(
\frac{N}{2p}-\frac{1}{2(q-1)}
\right)
\|V\|_{L^p(B_k)}^p
\nonumber\\
\leqslant
2\varepsilon\|V\|_{L^p(B_k)}^p
+
C_\varepsilon(N,q,p,U_0).
\label{Cpsi1:limit-energy}
\end{align}
Since \(p<N(q-1)\),
\[
\gamma_p
:=
\frac{N}{2p}-\frac{1}{2(q-1)}
>0.
\]
Choosing \(\varepsilon<\gamma_p/4\) and absorbing the
\(L^p\)-term into the left-hand side, we obtain
\begin{equation}\label{Cpsi1:weighted-estimate}
\left\|
 \nabla\bigl(|V|^{\frac p2-1}V\bigr)
\right\|_{L^2(B_k)}^2
+
\|V\|_{L^p(B_k)}^p
\leqslant
C(N,q,p,U_0).
\end{equation}
The constant in \eqref{Cpsi1:weighted-estimate} is independent of
\(A\) and \(k\).

Applying the \(W^{1,p}\)-estimate from Lemma~\ref{lemma_1}, we obtain
\[
\|V\|_{W^{1,p}_0(B_k)}^p
\leqslant
C(p)
\left[
\left\|
 \nabla\bigl(|V|^{\frac p2-1}V\bigr)
\right\|_{L^2(B_k)}^2
+
\|V\|_{L^p(B_k)}^p
\right].
\]
Thus, we conclude that
\[
\|V_k\|_{W^{1,p}_0(B_k)}
\leqslant
C(N,q,p,U_0),
\]
with a constant independent of both \(k\) and \(A\).

\medskip
\noindent
\textit{Case 2: \(N=1\).}
In this case the assumed range of \(q\) is
\(
2<q\leqslant3.
\)
Denote
\(
B=\|U_0\|_{L^\infty(\mathbb R)}.
\)
Because \(\psi_A'(s+u)=0\) whenever \(|s+u|>A\), the integrand
defining \(\partial_uP_\delta(v,u)\) vanishes unless
\(
|s|\leqslant A+B.
\)
Thus
\begin{align}
|\partial_uP_\delta(v,u)|
&\leqslant
C(p)qA^{q-1}
\int_{|s|\leqslant A+B}|s|^{p-2}\,ds
\nonumber\\
&\leqslant
C(p,q)A^{q-1}(A+B)^{p-1}.
\label{Cpsi1:N1-local}
\end{align}
By Lemma \ref{heat_L}, $U_0'\in L^1(\mathbb R)$ and consequently 
\[
\left|
\int_{B_k}
\partial_uP_\delta(V,U_0)U_0'\,dx
\right|
\leqslant
C(A,p,q,U_0)\|U_0'\|_{L^1(\mathbb R)}.
\]
Proceeding as above and using
\[
\frac{1}{2p}-\frac{1}{2(q-1)}>0,
\]
we obtain
\[
\|V_k\|_{W^{1,p}_0(B_k)}
\leqslant
C(A,p,q,U_0).
\]
This bound is independent of \(k\), but the estimate
\eqref{Cpsi1:N1-local} depends on \(A\). 

\medskip
\noindent
\textit{Case 3: \(N=2\) and \(q=2\).}
Here inequality \eqref{Cpsi1:Pu-uniform} becomes
\[
|\partial_uP_\delta(v,u)|
\leqslant
C(p)
\left(
|v|^p+|u||v|^{p-1}
\right).
\]
On the complement of a sufficiently large ball \(B_R\), the
decay of \(\nabla U_0\) allows the term
\[
C\int_{B_k\setminus B_R}
|\nabla U_0||V|^p\,dx
\]
to be absorbed into the left-hand side of the energy estimate.
The remaining term on \(B_k\setminus B_R\) is estimated by
\[
\varepsilon\|V\|_{L^p(B_k)}^p
+
C_\varepsilon
\|U_0\nabla U_0\|_{L^p(\mathbb R^2)}^p.
\]

On the fixed ball \(B_R\), however, one uses the local estimate
\[
|\partial_uP_\delta(v,u)|
\leqslant
C(A,p,\|U_0\|_\infty)|v|^{p-1}.
\]
Consequently,
\[
\int_{B_R}
|\partial_uP_\delta(V,U_0)||\nabla U_0|\,dx
\leqslant
\varepsilon\|V\|_{L^p(B_R)}^p
+
C_{\varepsilon,R}(A,p,U_0).
\]
The rest of the argument gives
\[
\|V_k\|_{W^{1,p}_0(B_k)}
\leqslant
C(A,p,U_0),
\]
independently of \(k\). 
\end{proof}

\subsection{Passage to the limit and solutions in the whole space}
We now prove the existence of a whole-space solution of
\eqref{ell_V:gen} with the truncated nonlinearity \eqref{psi1}.

\begin{proof}[Proof of Theorem \ref{existence_I}]
Recall that 
\(
U_0=e^\Delta u_0.
\)
For each $k\in\bbfN$, let $V_k$ be a solution of the Dirichlet
problem
\begin{equation}\label{Vk_equ:correct}
\begin{aligned}
-\Delta V_k-\frac{x}{2}\cdot\nabla V_k
-\frac{1}{2(q-1)}V_k
&=
-\partial_{x_1}\psi(V_k+U_0),
&&x\in B_k,\\
V_k&=0,
&&x\in\partial B_k.
\end{aligned}
\end{equation}
We extend each $V_k$ by zero outside $B_k$. Since
$V_k\in W^{1,p}_0(B_k)$, its zero extension, still denoted by
$V_k$, belongs to $W^{1,p}(\bbfR^N)$.

We consider separately the two ranges of $q$.

\medskip
\noindent
\textit{Case 1: $q>1+\frac2N$.}
By Lemma \ref{C_psi}, problem \eqref{Vk_equ:correct} has a
solution
\(
V_k\in W^{1,2}_0(B_k).
\)
satisfying 
 the estimate
\(
\|V_k\|_{W^{1,2}(\bbfR^N)}
\leqslant C,
\)
where $C$ is independent of $k$. Thus, after passing to a subsequence, there exists
\(
V\in W^{1,2}(\bbfR^N)
\)
such that
\begin{equation}\label{Vk:weak:W12}
V_k\rightharpoonup V
\qquad\text{weakly in }W^{1,2}(\bbfR^N).
\end{equation}
Furthermore, by the Rellich--Kondrachov theorem, after passing to
a further subsequence if necessary,
\begin{equation}\label{Vk:local:strong}
V_k\longrightarrow V
\qquad\text{strongly in }L^2_{\rm loc}(\bbfR^N)
\end{equation}
and
\begin{equation}\label{Vk:ae}
V_k(x)\longrightarrow V(x)
\qquad\text{for a.e. }x\in\bbfR^N.
\end{equation}

\medskip
\noindent
\textit{Case 2:
$q\in\left(1+\frac1N,1+\frac2N\right]$.}
Set
\[
p_-=
\max\left\{1,\frac{N(q-1)}{q}\right\},
\qquad
p_+=N(q-1).
\]
Thus
\(
1\leqslant p_-<p_+\leqslant2.
\)
Fix one exponent
\(
p_0\in(p_-,p_+)
\)
and, for each $k$, choose a solution
\(
V_k\in W^{1,p_0}_0(B_k)
\)
of problem \eqref{Vk_equ:correct} provided by Lemma \ref{lemma_2}.


Since $B_k$ is bounded, standard elliptic bootstrap applied to equation
\eqref{Vk_equ:correct} shows that each such $V_k$ belongs, in
particular, to
\(
W^{1,p}_0(B_k)
\)
for every finite $p>1$. Thus Lemma \ref{C_psi_1} can be applied
to the same solution $V_k$ for every
\(
p\in(p_-,p_+).
\)
Consequently,
\begin{equation}\label{Vk:all:p:bound}
\|V_k\|_{W^{1,p}(\bbfR^N)}
\leqslant C_p
\qquad
\text{for every }p\in(p_-,p_+),
\end{equation}
where $C_p$ is independent of $k$.

Let $\{p_j\}_{j=1}^\infty$ be a countable dense subset of
$(p_-,p_+)$. By reflexivity and a diagonal argument, there exist
a subsequence, still denoted by $\{V_k\}$, and functions
$V^{(j)}\in W^{1,p_j}(\bbfR^N)$ such that
\[
V_k\rightharpoonup V^{(j)}
\qquad\text{weakly in }W^{1,p_j}(\bbfR^N)
\]
for every $j$.
The
Rellich--Kondrachov theorem yields, after another diagonal
extraction,
\[
V_k\longrightarrow V
\qquad\text{strongly in }L^1_{\rm loc}(\bbfR^N)
\]
and
\[
V_k(x)\longrightarrow V(x)
\qquad\text{for a.e. }x\in\bbfR^N.
\]
Hence all the weak limits $V^{(j)}$ coincide with the same
function $V$. Therefore
\[
V\in W^{1,p_j}(\bbfR^N)
\qquad\text{for every }j.
\]

We claim that, in fact,
\begin{equation}\label{V:all:p}
V\in W^{1,p}(\bbfR^N)
\qquad\text{for every }p\in(p_-,p_+).
\end{equation}
Indeed, fix such a $p$. Choose $p_i,p_j$ from the dense sequence
so that
\(
p_i<p<p_j.
\)
Since
\(
V\in W^{1,p_i}(\bbfR^N)\cap W^{1,p_j}(\bbfR^N),
\)
the interpolation inequality applied both to $V$ and to
$\nabla V$ gives
\(
V\in W^{1,p}(\bbfR^N).
\)

It remains to show that the limit $V$ is a weak solution of
\eqref{ell_V:gen}. The argument is the same in both cases.

Fix
\(
\vf\in C^\infty_c(\bbfR^N).
\)
For all sufficiently large $k$,
\(
\supp\vf\subset B_k.
\)
The weak formulation of problem \eqref{Vk_equ:correct} therefore gives
\begin{equation}\label{Vk:weak:form}
\begin{split}
&
\int_{\bbfR^N}\nabla V_k\cdot\nabla\vf\ud x
-
\int_{\bbfR^N}
\frac{x}{2}\cdot\nabla V_k\,\vf\ud x
-
\frac{1}{2(q-1)}
\int_{\bbfR^N}V_k\vf\ud x
\\
&\qquad
=
\int_{\bbfR^N}
\psi(V_k+U_0)\partial_{x_1}\vf\ud x.
\end{split}
\end{equation}

Choose an exponent $p>1$ for which the sequence $\{V_k\}$ is
bounded in $W^{1,p}(\bbfR^N)$; in Case 1 we take $p=2$, while in
Case 2 we may take any $p\in(p_-,p_+)$. The weak convergence in
$W^{1,p}(\bbfR^N)$ implies
\[
\int_{\bbfR^N}\nabla V_k\cdot\nabla\vf\ud x
\longrightarrow
\int_{\bbfR^N}\nabla V\cdot\nabla\vf\ud x,
\]
and, since $x\vf$ has compact support,
\[
\int_{\bbfR^N}
\frac{x}{2}\cdot\nabla V_k\,\vf\ud x
\longrightarrow
\int_{\bbfR^N}
\frac{x}{2}\cdot\nabla V\,\vf\ud x.
\]
Similarly,
\[
\int_{\bbfR^N}V_k\vf\ud x
\longrightarrow
\int_{\bbfR^N}V\vf\ud x.
\]

On the other hand, by the almost everywhere convergence
\eqref{Vk:ae}, or its counterpart in Case 2, and the continuity of
$\psi$,
\[
\psi(V_k+U_0)
\longrightarrow
\psi(V+U_0)
\qquad\text{a.e. in }\bbfR^N.
\]
Moreover,
\[
|\psi(V_k+U_0)\partial_{x_1}\vf|
\leqslant
\|\psi\|_{L^\infty(\bbfR)}
|\partial_{x_1}\vf|,
\]
and the right-hand side is integrable because $\vf$ has compact
support. Hence, by the Lebesgue dominated convergence theorem,
\[
\int_{\bbfR^N}
\psi(V_k+U_0)\partial_{x_1}\vf\ud x
\longrightarrow
\int_{\bbfR^N}
\psi(V+U_0)\partial_{x_1}\vf\ud x.
\]

Passing to the limit in \eqref{Vk:weak:form}, we conclude that
$V$ is a weak solution of equation \eqref{ell_V:gen}, namely, it satisfies relation
\eqref{vf_V:gen}.
\end{proof}

\section{Proof of the main result}

\subsection{Truncation-independent estimates} \label{subsection:truncation:independent}
In the following theorem, we show that solutions to the elliptic equation \eqref{ell_V:gen} with the truncated nonlinearity \eqref{psi1} are uniformly bounded by a constant independent of the truncation level $A > 0$.

\begin{theorem}\label{thm:uniform:estimates}
Assume that $N\geqslant1$ and
\[
1+\frac1N<q
\quad\text{if }N\in\{1,2\},
\qquad
1+\frac1N<q<\frac{N}{N-2}
\quad\text{if }N\geqslant3.
\]
Let $\psi=\psi_A$ be the truncated nonlinearity defined by
\eqref{psi1}, with an arbitrary truncation level $A>0$. Let
\(
U_0=e^\Delta u_0,
\)
where the homogeneous function $u_0$ is given by \eqref{u0}. Let
$V$ be the whole-space weak solution of \eqref{ell_V:gen}
constructed in Theorem \ref{existence_I}. Then
\(
V\in L^\infty(\bbfR^N)
\)
and
\begin{equation}\label{V:Linfty}
\|V\|_{L^\infty(\bbfR^N)}
\leqslant C(N,q,U_0),
\end{equation}
where the constant is independent of the truncation level $A$.
\end{theorem}

\begin{proof}
Let $B_k$ be the approximating balls and let $V_k$ be the
corresponding solutions of the bounded-domain problems 
\eqref{Vk_equ}
used in
the proof of Theorem \ref{existence_I}. 
After passing to a subsequence, these approximations converge
locally to the whole-space solution $V$ constructed in Theorem
\ref{existence_I}. 
We now prove estimates for $V_k$ that are independent
of both $k$ and the truncation level $A>0$.
In the following, for each $k$, we use the function 
\(
U_k=V_k+U_0.
\)
Since $U_0$ satisfies the linear elliptic equation \eqref{eq:U0:A} below,
the function $U_k$ solves
\begin{equation}\label{Uk:bounded:theorem}
\begin{aligned}
-\Delta U_k-\frac{x}{2}\cdot\nabla U_k
-\frac{1}{2(q-1)}U_k
+\partial_{x_1}\psi_A(U_k)&=0
&&\text{in }B_k,\\
U_k&=U_0
&&\text{on }\partial B_k.
\end{aligned}
\end{equation}

We distinguish four cases.

\medskip
\noindent
\textit{Case 1: $N=1$.}
The condition on $q$ becomes $q>2$. Applying
Lemma \ref{lem:uniform-Linfty-N1} (see below) to $V_k$, we obtain
\[
\|V_k\|_{L^\infty(B_k)}
\leqslant C(q,U_0),
\]
where the constant is independent of $k$ and $A$. Passing to the
whole-space limit gives inequality \eqref{V:Linfty}.

\medskip
\noindent
\textit{Case 2: $N\geqslant2$ and
$q>1+\frac{1}{N-1}$.}
The bounded-domain solutions have the regularity required in
Lemma \ref{lem:Linfty}. Applying that lemma to $U_k$, which solves
\eqref{Uk:bounded:theorem}, gives
\[
\|U_k\|_{L^\infty(B_k)}
\leqslant C(N,q,U_0),
\]
where the constant is independent of $k$ and of the nonlinearity
$\psi_A$, and hence independent of $A$.
The heat-profile estimate in Lemma \ref{heat_L} gives
\(
\|U_0\|_{L^\infty(\bbfR^N)}
\leqslant C(N,q,U_0).
\)
Therefore,
\[
\|V_k\|_{L^\infty(B_k)}
\leqslant
\|U_k\|_{L^\infty(B_k)}
+\|U_0\|_{L^\infty(\bbfR^N)}
\leqslant C(N,q,U_0).
\]
Passing to the limit yields
\[
\|V\|_{L^\infty(\bbfR^N)}
\leqslant C(N,q,U_0).
\]

\medskip
\noindent
\textit{Case 3: $N\geqslant2$, $q<2$, and
$q\leqslant1+\frac{1}{N-1}$.}
Under the assumptions of the theorem, this remaining range
satisfies
\[
1+\frac1N<q\leqslant1+\frac2N,
\qquad q<2.
\]
Indeed, if $N\geqslant3$, then
\[
1+\frac{1}{N-1}<1+\frac2N<2,
\]
whereas for $N=2$ the inequality $q<2$ is part of the present case.

Choose, once and for all,
\begin{equation}\label{p:choice:uniform:theorem}
\max\left\{
1,\frac{N(q-1)}q
\right\}
<p<N(q-1).
\end{equation}
Such a choice is possible because $q>1+1/N$. By the
truncation-independent part of Lemma \ref{C_psi_1},
\begin{equation}\label{Vk:uniform:W1p:theorem}
\|V_k\|_{W^{1,p}(B_k)}
\leqslant C(N,q,p,U_0),
\end{equation}
where the constant is independent of $k$ and $A$.

Extend $V_k$ by zero outside $B_k$. Since
$V_k\in W^{1,p}_0(B_k)$, its zero extension belongs to
$W^{1,p}(\bbfR^N)$, and \eqref{Vk:uniform:W1p:theorem} remains
valid on $\bbfR^N$. By weak compactness and the convergence used
in Theorem \ref{existence_I}, we obtain
\begin{equation}\label{V:uniform:W1p:theorem}
\|V\|_{W^{1,p}(\bbfR^N)}
\leqslant C(N,q,p,U_0).
\end{equation}

Because
\(
p<N(q-1)\leqslant2\leqslant N,
\)
we have $p<N$. The Sobolev inequality therefore gives
\begin{equation}\label{V:pstar:theorem}
\|V\|_{L^{p^*}(\bbfR^N)}
\leqslant
C(N,p)\|V\|_{W^{1,p}(\bbfR^N)}
\leqslant C(N,q,p,U_0),
\end{equation}
where
\(
p^*=\frac{Np}{N-p}.
\)
Moreover,
\(p^*>N(q-1).
\)
Indeed, by a direct calculation
\begin{align*}
p^*>N(q-1) \quad \text{is equivalent to} \quad
p>\frac{N(q-1)}q,
\end{align*}
which follows from \eqref{p:choice:uniform:theorem}.

By Lemma \ref{heat_L},
\(
U_0\in L^{p^*}(\bbfR^N)
\)
because $p^*>N(q-1)$. Consequently, for
\(
U=V+U_0,
\)
we have
\begin{equation}\label{U:pstar:theorem}
\|U\|_{L^{p^*}(\bbfR^N)}
\leqslant
\|V\|_{L^{p^*}(\bbfR^N)}
+\|U_0\|_{L^{p^*}(\bbfR^N)}
\leqslant C(N,q,p,U_0),
\end{equation}
uniformly with respect to $A$.
Now, 
Lemma \ref{lem:large:p:U}, applied with $p_0=p^*$, yields
\[
\|U\|_{L^\infty(\bbfR^N)}
\leqslant
C(N,q,p^*)\|U\|_{L^{p^*}(\bbfR^N)}
\leqslant C(N,q,U_0).
\]
The constant in Lemma \ref{lem:large:p:U} is independent of the
nonlinearity and therefore independent of $A$. Finally,
\[
\|V\|_{L^\infty(\bbfR^N)}
\leqslant
\|U\|_{L^\infty(\bbfR^N)}
+\|U_0\|_{L^\infty(\bbfR^N)}
\leqslant C(N,q,U_0).
\]

\medskip
\noindent
\textit{Case 4: $N=2$ and $q=2$.} Use
Lemma~\ref{lem:uniform:Linfty:N2:q2}, below.
\end{proof}

The following lemmas were used in the proof of Theorem \ref{thm:uniform:estimates}. We begin with the truncation independent estimates in one dimension.

\begin{lemma}[Uniform one-dimensional \(L^\infty\) estimate]
\label{lem:uniform-Linfty-N1}
Let \(q>2\) and let \(\Omega=(a,b)\subset\mathbb R\). Suppose that
\(u_0:\mathbb R\setminus\{0\}\to\mathbb R\) is homogeneous of
degree \(-1/(q-1)\)
and set $U_0=e^\Delta u_0$.
Let $\psi$ be the truncated nonlinearity 
\eqref{psi1} with an arbitrary truncation level $A>0$.
Suppose that \(V\in W^{1,2}_0(a,b)\) is a weak solution of the ordinary differential equation
\begin{equation}
\label{ell:bounded:N1}
-V''-\frac{x}{2}V'
-\frac{1}{2(q-1)}V
+\frac{d}{dx}\psi_A(V+U_0)=0
\qquad\text{in }(a,b).
\end{equation}
Then
\[
\|V\|_{L^\infty(a,b)}
\leqslant C(q,u_0),
\]
where the constant is independent of \(A\) and of the interval
\((a,b)\).
\end{lemma}

\begin{proof}
Set
\[
c=\frac{1}{2(q-1)},
\qquad
\beta=\frac12-c>0,
\qquad
B=\|U_0\|_{L^\infty(\mathbb R)},
\]
and define
\[
K=\frac{q}{\beta}
\int_{\mathbb R}|U_0|^{q-1}|U_0'|\,dx.
\]
The heat-profile estimates imply that $B,K<\infty$.
Indeed,
\[
|U_0(x)|^{q-1}|U_0'(x)|
\leqslant
C(q,u_0)(1+|x|)^{-q/(q-1)},
\]
and $q/(q-1)>1$.

Since $\psi_A$ is globally Lipschitz, the equation gives
$V\in W^{2,2}(a,b)\subset C^1([a,b])$.
We first prove that
\begin{equation}
\label{N1:uniform:L1}
\|V\|_{L^1(a,b)}\leqslant K.
\end{equation}
Let $I=(\ell,r)$ be a connected component of $\{V\neq0\}$,
and let $\sigma\in\{-1,1\}$ be the sign of $V$ on $I$.
Integrating \eqref{ell:bounded:N1} over $I$ and multiplying
by $\sigma$, we obtain
\[
-\sigma[V']_\ell^r
+\beta\int_I|V|\,dx
=
-\sigma[\psi_A(U_0)]_\ell^r.
\]
Here $V(\ell)=V(r)=0$ and
$-\sigma[V']_\ell^r\geqslant0$. Consequently,
\[
\beta\int_I|V|\,dx
\leqslant
\int_I|(\psi_A(U_0))'|\,dx.
\]
Summing over the components and using
$|\psi_A'(s)|\leqslant q|s|^{q-1}$ proves
\eqref{N1:uniform:L1}.

Now set $U=V+U_0$. Since
\[
-U_0''-\frac{x}{2}U_0'-cU_0=0,
\]
we have
\[
-U''-\frac{x}{2}U'-cU+(\psi_A(U))'=0
\qquad\text{in }(a,b),
\]
with $U=U_0$ at the endpoints.
Define
\[
w=(|U|-B)_+.
\]
Then $w\in W^{1,2}_0(a,b)$ and $0\leqslant w\leqslant|V|$,
so \eqref{N1:uniform:L1} gives
\[
\|w\|_{L^1(a,b)}\leqslant K.
\]

Test the equation for $U$ with $\operatorname{sgn}(U)w$.
The convection term vanishes: its integrand is the derivative
of $P_A(U)$, where
\[
P_A(s)=
\int_0^s
\psi_A'(t)\operatorname{sgn}(t)(|t|-B)_+\,dt,
\]
and $P_A(U(a))=P_A(U(b))=0$.
Integration by parts therefore yields
\[
\|w'\|_2^2+
\left(\frac14-c\right)\|w\|_2^2
=
cB\|w\|_1,
\]
where all norms are taken over $(a,b)$.
Writing 
$$M=\|w\|_\infty,$$ we deduce that
\begin{equation}
\label{N1:excess:energy}
\|w'\|_2^2
\leqslant c\|w\|_2^2+cB\|w\|_1
\leqslant cK(M+B).
\end{equation}

Finally, since $w(a)=0$, the fundamental theorem of calculus
and the Cauchy--Schwarz inequality give
\[
w(x)^{3/2}
=
\frac32\int_a^x w^{1/2}w'\,dy
\leqslant
\frac32\|w\|_1^{1/2}\|w'\|_2.
\]
Taking the supremum and using \eqref{N1:excess:energy},
we obtain
\[
M^3
\leqslant
\frac94K\|w'\|_2^2
\leqslant
\frac{9c}{4}K^2(M+B).
\]
Thus either $M\leqslant B$, or
$M\leqslant3\sqrt{c/2}\,K$.
Since $|V|\leqslant|U|+B\leqslant w+2B$, it follows that
\[
\|V\|_{L^\infty(a,b)}
\leqslant
3B+3\sqrt{\frac c2}\,K
=
C(q,u_0).
\]
The right-hand side is independent of $A$ and $(a,b)$.
\end{proof}

\begin{lemma}
\label{lem:Linfty}
Assume that $N\geqslant2$ and
\begin{equation}
\label{q:maximum}
q>1+\frac{1}{N-1},
\qquad\text{or equivalently}\qquad
\frac{1}{q-1}<N-1.
\end{equation}
Let $\Omega\subset\bbfR^N$ be a bounded domain with $C^2$ boundary,
and let
\(
U_0=e^\Delta u_0,
\)
where $u_0$ is given by \eqref{u0}. Let
$\psi\in W^{1,\infty}(\bbfR)$ be arbitrary.
Suppose that $U\in W^{1,2}(\Omega)$ is a weak solution of
\begin{equation}
\label{U:bounded}
-\Delta U-\frac{x}{2}\cdot\nabla U- \frac{1}{2(q-1)}U
+
\partial_{x_1}\psi(U)=0
\qquad\text{in }\Omega,
\end{equation}
with boundary condition
\begin{equation}
\label{U:boundary}
U=U_0
\qquad\text{on }\partial\Omega
\end{equation}
in the sense of traces. Then
\begin{equation}
\label{U:Linfty}
\|U\|_{L^\infty(\Omega)}
\leqslant C(N,q,U_0),
\end{equation}
where the constant is independent of $\Omega$ and $\psi$.
%
\end{lemma}

\begin{proof}

Set
\[
a=\frac{1}{q-1},
\qquad
c=\frac{a}{2}=\frac{1}{2(q-1)}.
\]
By the Sobolev chain rule,
\(
\partial_{x_1}\psi(U)
=
\psi'(U)\partial_{x_1}U
\)
a.e.~in $\Omega$.
Consequently, equation \eqref{U:bounded} can be written as
\begin{equation}
\label{U:linearized}
-\Delta U+
\left(
\psi'(U)e_1-\frac{x}{2}
\right)\cdot\nabla U
-cU=0
\qquad\text{in }\Omega.
\end{equation}

Write
\(
x=(x_1,x'),
\)
where
\(
x'\in\bbfR^{N-1},
\)
and define
\begin{equation}
\label{Phi:definition}
\Phi(x')
=
e^{\Delta_{\bbfR^{N-1}}}|x'|^{-a}
=
\int_{\bbfR^{N-1}}
G_{N-1}(x'-y,1)|y|^{-a}\,\ud y,
\end{equation}
where $G_{N-1}$ denotes the heat kernel for the heat equation on $\bbfR^{N-1}.$
Condition \eqref{q:maximum} gives
\[
0<a<N-1.
\]
Thus, $|x'|^{-a}$ is locally integrable in
$\bbfR^{N-1}$, and, by Lemma \ref{heat_L}, the function $\Phi$ is smooth, strictly positive, and
bounded. Moreover, by Lemma~\ref{heat_L}, it satisfies the linear equation
\begin{equation}
\label{Phi:reduced-equation}
-\Delta_{x'}\Phi
-\frac{x'}{2}\cdot\nabla_{x'}\Phi
-\frac{a}{2}\Phi=0
\qquad\text{in }\bbfR^{N-1}.
\end{equation}
We extend $\Phi$ to $\bbfR^N$ by setting
\(
\Phi(x)=\Phi(x').
\)
Since
\(
\partial_{x_1}\Phi=0
\)
and
\(
c=\frac a2,
\)
equation \eqref{Phi:reduced-equation} implies
\begin{equation}
\label{Phi:equation}
-\Delta\Phi+
\left(
\psi'(U)e_1-\frac{x}{2}
\right)\cdot\nabla\Phi
-c\Phi=0
\qquad\text{in }\Omega.
\end{equation}
The term containing $\psi'(U)$ vanishes because $\Phi$ is
independent of $x_1$.

We next recall the two-sided estimate
\begin{equation}
\label{Phi:two-sided}
\frac{c_0}{1+|x'|^a}
\leqslant
\Phi(x')
\leqslant
\frac{C_0}{1+|x'|^a},
\qquad x'\in\bbfR^{N-1},
\end{equation}
where $c_0,C_0>0$ depend only on $N$ and $q$. 
These estimates follow from the heat-profile asymptotics recalled in
Lemma~\ref{L:A3} below.
The heat-profile estimate recalled in Lemma \ref{heat_L} and applied to $U_0=e^\Delta u_0$ gives 
\[
|U_0(x)|
\leqslant
\frac{C(U_0)}{1+|x|^a}
\leqslant
\frac{C(U_0)}{1+|x'|^a}.
\]
Combining this estimate with the lower bound in
\eqref{Phi:two-sided}, we obtain
\begin{equation}
\label{U0:Phi}
|U_0(x)|
\leqslant
M_0\Phi(x')
\qquad\text{for all }x\in\bbfR^N,
\end{equation}
where $M_0=M_0(N,q,U_0)>0$. In particular, $M_0$ is independent
of $\Omega$ and $\psi$.

Now we define
\[
Z=\frac{U}{\Phi}.
\]
Since $\Omega$ is bounded and $\Phi$ is smooth and strictly
positive on $\overline\Omega$, the function $1/\Phi$ is bounded on
$\Omega$,
and hence $Z\in W^{1,2}(\Omega)$. Substituting $U=\Phi Z$ into
\eqref{U:linearized} and using \eqref{Phi:equation}, we obtain
\begin{equation}
\label{Z:equation}
-\Delta Z+b(x)\cdot\nabla Z=0
\qquad\text{in }\Omega,
\end{equation}
where
\begin{equation}\label{b:in:Z:equation}
b(x)
=
\psi'(U)e_1-\frac{x}{2}
-2\frac{\nabla\Phi}{\Phi}.
\end{equation}
This is a standard technique commonly used when applying the maximum principle to elliptic equations with a zero-order term (see, e.g., \cite[Chap.~2, Sec.~5]{PW84}); we therefore omit the calculations leading to equation \eqref{Z:equation}.
Note that, for each fixed $\Omega$ and $\psi$, we have
\(
b\in L^\infty(\Omega),
\)
because $\psi'\in L^\infty(\bbfR)$, $\Omega$ is bounded, and
$\nabla\Phi/\Phi$ is bounded on $\overline\Omega$.

Both \(Z\) and \(-Z\) are sufficiently regular solutions of the linear elliptic equation \eqref{Z:equation}. Therefore, the weak maximum principle can be applied to obtain
\begin{equation}
\label{Z:maximum}
\|Z\|_{L^\infty(\Omega)}
\leqslant
\|Z\|_{L^\infty(\partial\Omega)}.
\end{equation}
By the boundary condition \eqref{U:boundary},
\[
Z=\frac{U_0}{\Phi}
\qquad\text{on }\partial\Omega.
\]
Therefore, estimate \eqref{U0:Phi} gives
\(
|Z|\leqslant M_0
\)
on $\partial\Omega$.
It follows from \eqref{Z:maximum} that
\[
|Z(x)|\leqslant M_0
\qquad\text{for a.e. }x\in\Omega.
\]
Consequently,
\[
|U(x)|
\leqslant
M_0\Phi(x')
\qquad\text{for a.e. }x\in\Omega.
\]
Using the upper bound in \eqref{Phi:two-sided}, we conclude that
\[
\|U\|_{L^\infty(\Omega)}
\leqslant
M_0C_0
=
C(N,q,U_0).
\]
This estimate is independent of the domain $\Omega$ and of the nonlinearity $\psi$.
\end{proof}

The following lemma bootstraps a sufficiently high $L^{p_0}$ bound to an
$L^\infty$ bound. The estimate is independent of the nonlinearity $\psi$.

\begin{lemma}\label{lem:large:p:U}
Let $N\geqslant1$, $q>1$, and let
$\psi\in W^{1,\infty}(\bbfR)$ be arbitrary. Suppose that $U$ is
a sufficiently regular solution of
\begin{equation}\label{ell_U:large:p}
-\Delta U
-\frac12 x\cdot\nabla U
-\frac{1}{2(q-1)}U
+\partial_{x_1}\psi(U)=0
\qquad\text{in }\bbfR^N,
\end{equation}
such that the integrations by parts below are justified. Assume
that
\[
U\in L^{p_0}(\bbfR^N)
\qquad\text{for some}\qquad
p_0>\max\{1,N(q-1)\}.
\]
Then
there exists a constant $C=C(N,q,p_0)>0$, independent of
$\psi$, such that
\begin{equation}\label{U:Linfty:large:p}
\|U\|_{L^\infty(\bbfR^N)}
\leqslant
C(N,q,p_0)\|U\|_{L^{p_0}(\bbfR^N)}.
\end{equation}
\end{lemma}

\begin{proof}
Let $p>1$ and assume initially that
$U\in L^p(\bbfR^N)$. Testing equation \eqref{ell_U:large:p} with
$|U|^{p-2}U$ gives
\begin{equation}\label{large:p:energy}
\frac{4(p-1)}{p^2}
\left\|\nabla |U|^{p/2}\right\|_2^2
+
\left(
\frac{N}{2p}-\frac{1}{2(q-1)}
\right)\|U\|_p^p
=0.
\end{equation}
In this calculation, we use the following two identities obtained by
integration by parts:
\[
\int_{\bbfR^N}
 \partial_{x_1}\psi(U)|U|^{p-2}U\,\ud x
=0
\]
and
\[
-\frac12
\int_{\bbfR^N}
x\cdot\nabla U\,|U|^{p-2}U\,\ud x
=
-\frac{1}{2p}
\int_{\bbfR^N}x\cdot\nabla(|U|^p)\,\ud x
=
\frac{N}{2p}\|U\|_p^p.
\]
This calculation may be justified by first using regularized
powers and compactly supported cutoff functions and then passing
to the limit. Notice that
no estimate involving $\psi$ remains in the energy
identity \eqref{large:p:energy}.






If $p>N(q-1)$, set
\[
\beta_p
=
\frac{1}{2(q-1)}-\frac{N}{2p}>0.
\]
Then equation \eqref{large:p:energy} becomes
\begin{equation}\label{large:p:gradient:exact}
\left\|\nabla |U|^{p/2}\right\|_2^2
=
D_p\|U\|_p^p,
\qquad
D_p
=
\frac{p^2}{4(p-1)}\beta_p.
\end{equation}
Moreover, for $p\geqslant2$,
\[
D_p
\leqslant
\frac{p^2}{8(q-1)(p-1)}
\leqslant
\frac{p}{4(q-1)},
\]
which proves 
\begin{equation}\label{large:p:grad}
\left\|\nabla |U|^{p/2}\right\|_{L^2(\bbfR^N)}^2
\leqslant
C(q)p\|U\|_{L^p(\bbfR^N)}^p.
\end{equation}

We now bootstrap the integrability of $U$.

\medskip
\noindent
\textit{Case 1: $N\geqslant3$.}
Set
\[
\kappa=\frac{N}{N-2}.
\]
Applying the Sobolev inequality to
$w=|U|^{p/2}$ and using
\eqref{large:p:gradient:exact}, we obtain
\[
\|U\|_{\kappa p}^p
=
\left\||U|^{p/2}\right\|_{\frac{2N}{N-2}}^2
\leqslant
C(N)\left\|\nabla |U|^{p/2}\right\|_2^2
=
C(N)D_p\|U\|_p^p.
\]
Hence
\begin{equation}\label{large:p:iteration:step}
\|U\|_{\kappa p}
\leqslant
\bigl(C(N)D_p\bigr)^{1/p}\|U\|_p.
\end{equation}

Let
\(
p_j=\kappa^jp_0.
\)
Since $p_j\geqslant p_0>1$,
\[
D_{p_j}\leqslant C(q,p_0)p_j.
\]
Iterating \eqref{large:p:iteration:step} yields
\[
\|U\|_{p_m}
\leqslant
\prod_{j=0}^{m-1}
\bigl(C(N)D_{p_j}\bigr)^{1/p_j}
\|U\|_{p_0}.
\]
The product is bounded uniformly in $m$ because
\[
\sum_{j=0}^\infty
\frac{\log\bigl(\max\{1,C(N)D_{p_j}\}\bigr)}{p_j}
\leqslant
C\sum_{j=0}^\infty
\frac{1+\log p_j}{p_j}
<\infty.
\]
Since $p_m\to\infty$, it follows that
\(
\|U\|_\infty
\leqslant
C(N,q,p_0)\|U\|_{p_0}.
\)

\medskip
\noindent
\textit{Case 2: $N=2$.}
The Gagliardo--Nirenberg inequality gives
\[
\|w\|_4^2
\leqslant
C\|w\|_2\|\nabla w\|_2.
\]
Taking $w=|U|^{p/2}$ and using
\eqref{large:p:gradient:exact}, we obtain
\[
\|U\|_{2p}^p
\leqslant
C\|U\|_p^{p/2}
\left\|\nabla |U|^{p/2}\right\|_2
=
CD_p^{1/2}\|U\|_p^p.
\]
Therefore,
\[
\|U\|_{2p}
\leqslant
\bigl(CD_p^{1/2}\bigr)^{1/p}\|U\|_p.
\]
Taking $p_j=2^jp_0$ and iterating as above gives
\[
\|U\|_\infty
\leqslant
C(q,p_0)\|U\|_{p_0}.
\]

\medskip
\noindent
\textit{Case 3: $N=1$.}
Using the inequality
\(
\|w\|_\infty^2
\leqslant
C\|w\|_2\|w'\|_2
\)
with $w=|U|^{p_0/2}$, we obtain
\begin{align*}
\|U\|_\infty^{p_0}
\leqslant
C\|U\|_{p_0}^{p_0/2}
\left\|\partial_x|U|^{p_0/2}\right\|_2
=
CD_{p_0}^{1/2}\|U\|_{p_0}^{p_0}.
\end{align*}
Consequently,
\[
\|U\|_\infty
\leqslant
\bigl(CD_{p_0}^{1/2}\bigr)^{1/p_0}
\|U\|_{p_0}
\leqslant
C(q,p_0)\|U\|_{p_0}.
\]
%
\end{proof}


The following lemma improves truncation-independent estimates in the remaining case
$N=2$ and $q=2$. 

\begin{lemma}
\label{lem:uniform:Linfty:N2:q2}
Let $N=2$, $q=2$, and let $u_0$ satisfy \eqref{u0}.
Set
\(
U_0=e^\Delta u_0.
\)
Let $\psi_A$ be the truncated nonlinearity defined in
\eqref{psi1}, with $A>0$, and let $B_k=B(0,k)$.
Fix $p\in(1,2)$ and suppose that
$V_k\in W^{1,p}_0(B_k)$ is a weak solution of
\begin{equation}
\label{endpoint:N2:q2:bounded}
\begin{aligned}
-\Delta V_k-\frac{x}{2}\cdot\nabla V_k
-\frac12 V_k
+\partial_{x_1}\psi_A(V_k+U_0)&=0
&&\text{in }B_k,\\
V_k&=0
&&\text{on }\partial B_k.
\end{aligned}
\end{equation}
Then
\begin{equation}
\label{endpoint:N2:q2:W1p}
\|V_k\|_{W^{1,p}(B_k)}
\leqslant C(p,U_0),
\end{equation}
where the constant is independent of both $k$ and $A$.

Consequently, if $V$ is the whole-space solution constructed
in Theorem~\ref{existence_I} as a limit of the solutions $V_k$,
then
\begin{equation}
\label{endpoint:N2:q2:Linfty}
\|V\|_{L^\infty(\bbfR^2)}
\leqslant C(U_0),
\end{equation}
where the constant is independent of $A$.
\end{lemma}

\begin{proof}
Write $V=V_k$ and set
\[
D=\|\partial_{x_1}U_0\|_{L^\infty(\bbfR^2)},
\qquad
F=|U_0\partial_{x_1}U_0|.
\]
By Lemma~\ref{heat_L},
\[
|U_0(x)|+|\nabla U_0(x)|
\leqslant \frac{C(U_0)}{1+|x|}.
\]
Hence $D<\infty$ and
\[
F(x)\leqslant \frac{C(U_0)}{(1+|x|)^2},
\qquad
F\in L^s(\bbfR^2)
\quad\text{for every }s>1.
\]

For each fixed $A$ and $k$, the boundedness of $\psi_A'$
and elliptic bootstrap give
$V\in W^{2,t}(B_k)$ for every finite $t>1$.
Thus the regularized power tests below are admissible.
The constants in this preliminary regularity argument
need not be uniform in $A$ or $k$.

For $1<s<2$, set
\[
w_s=|V|^{s/2-1}V,
\qquad
\alpha_s=\frac{4(s-1)}{s^2},
\qquad
\beta_s=\frac1s-\frac12.
\]
Unless otherwise indicated, norms of $V$ and $w_s$ are
taken over $B_k$.

Testing \eqref{endpoint:N2:q2:bounded} with
\[
\Phi_{s,\delta}(V)
=
V(V^2+\delta)^{(s-2)/2},
\]
integrating the convection term by parts, and letting
$\delta\downarrow0$, we obtain
\begin{align}
\alpha_s\|\nabla w_s\|_2^2
+\beta_s\|V\|_s^s
&\leqslant
\int_{B_k}|\partial_{x_1}U_0|
\left|
(s-1)\int_0^V
\psi_A'(z+U_0)|z|^{s-2}\,\ud z
\right|\ud x.
\label{endpoint:power:energy}
\end{align}
To justify the nonlinear integration by parts, use the
primitive
\[
P_{s,\delta}(v,u)
=
\int_0^v
\psi_A(z+u)\Phi_{s,\delta}'(z)\,\ud z.
\]
Since $P_{s,\delta}(0,u)=0$ and $V=0$ on $\partial B_k$,
\[
\int_{B_k}
\psi_A(V+U_0)\Phi_{s,\delta}'(V)
\partial_{x_1}V\,\ud x
=
-\int_{B_k}
\partial_uP_{s,\delta}(V,U_0)
\partial_{x_1}U_0\,\ud x.
\]
Dominated convergence applies to the latter integral,
while Fatou's lemma gives the diffusion term in
\eqref{endpoint:power:energy}.

The bound $|\psi_A'(z)|\leqslant 2|z|$ gives
\begin{align*}
\left|
(s-1)\int_0^v
\psi_A'(z+u)|z|^{s-2}\,\ud z
\right|
&\leqslant
2(s-1)\int_0^{|v|}
(z+|u|)z^{s-2}\,\ud z\\
&=
\frac{2(s-1)}s|v|^s
+2|u||v|^{s-1}.
\end{align*}
Consequently, by H\"older's inequality,
\begin{equation}
\label{endpoint:basic:estimate}
\alpha_s\|\nabla w_s\|_2^2
+\beta_s\|V\|_s^s
\leqslant
\frac{2(s-1)}s D\|V\|_s^s
+
2\|F\|_{L^s(\bbfR^2)}\|V\|_s^{s-1}.
\end{equation}
The factor $s-1$ in the first term on the right-hand side
allows us to absorb this term by choosing $s$ close to $1$.

Choose an auxiliary exponent
\[
1<r<
\min\left\{
p,\,
1+\frac{1}{1+4D}
\right\}.
\]
Then
\[
\gamma_r
:=
\beta_r-\frac{2(r-1)}rD
=
\frac{2-r-4D(r-1)}{2r}
>0.
\]
Applying \eqref{endpoint:basic:estimate} with $s=r$
and using Young's inequality, we obtain
\begin{align*}
\alpha_r\|\nabla w_r\|_2^2
+\gamma_r\|V\|_r^r
&\leqslant
2\|F\|_{L^r(\bbfR^2)}\|V\|_r^{r-1}\\
&\leqslant
\frac{\gamma_r}{2}\|V\|_r^r+C(p,U_0).
\end{align*}
Therefore,
\[
\|\nabla w_r\|_2^2+\|V\|_r^r
\leqslant C(p,U_0).
\]
Lemma~\ref{lemma_1} now yields
\begin{equation}
\label{endpoint:auxiliary:W1r}
\|V\|_{W^{1,r}(B_k)}
\leqslant C(p,U_0),
\end{equation}
uniformly in $A$ and $k$.

Extend $V$ by zero outside $B_k$.
Since $V\in W^{1,r}_0(B_k)$, the Sobolev inequality
on $\bbfR^2$ and \eqref{endpoint:auxiliary:W1r} imply
\[
\|V\|_{L^{r^*}(\bbfR^2)}
\leqslant C(p,U_0),
\qquad
r^*=\frac{2r}{2-r}>2.
\]
Because $r<p<2<r^*$, interpolation gives
\begin{equation}
\label{endpoint:uniform:Lp}
\|V\|_{L^p(B_k)}
\leqslant C(p,U_0).
\end{equation}

Returning to \eqref{endpoint:basic:estimate} with $s=p$
and using \eqref{endpoint:uniform:Lp}, we obtain
\[
\|\nabla w_p\|_2^2
\leqslant C(p,U_0).
\]
A second application of Lemma~\ref{lemma_1} gives
\[
\|V\|_{W^{1,p}(B_k)}^p
\leqslant
C(p)\left(
\|\nabla w_p\|_2^2+\|V\|_p^p
\right)
\leqslant C(p,U_0),
\]
which proves \eqref{endpoint:N2:q2:W1p}.

Finally, apply the estimate with $p=3/2$.
After extending $V_k$ by zero and passing to the
whole-space limit, weak lower semicontinuity gives
\[
\|V\|_{W^{1,3/2}(\bbfR^2)}
\leqslant C(U_0).
\]
The Sobolev embedding
$W^{1,3/2}(\bbfR^2)\hookrightarrow L^6(\bbfR^2)$
therefore yields
\[
\|V\|_{L^6(\bbfR^2)}
\leqslant C(U_0).
\]
Since $U_0\in L^6(\bbfR^2)\cap L^\infty(\bbfR^2)$,
the whole-space profile $U=V+U_0$ satisfies
\[
\|U\|_{L^6(\bbfR^2)}
\leqslant C(U_0).
\]
Applying Lemma~\ref{lem:large:p:U} to $U$ with $p_0=6$,
we obtain
\[
\|U\|_{L^\infty(\bbfR^2)}
\leqslant
C\|U\|_{L^6(\bbfR^2)}
\leqslant C(U_0),
\]
where the constant is independent of $\psi_A$ and hence
of $A$. Thus
\[
\|V\|_{L^\infty(\bbfR^2)}
\leqslant
\|U\|_{L^\infty(\bbfR^2)}
+\|U_0\|_{L^\infty(\bbfR^2)}
\leqslant C(U_0).
\]
\end{proof}

\subsection{Proof of Theorem \ref{main}}

For $A>0$, let $\psi_A$ be the truncated nonlinearity defined in
\eqref{psi1}. By Theorem \ref{existence_I}, there exists a weak
solution
\(
U_A=U_0+V_A
\)
of the equation
\begin{equation}\label{ell:truncated}
-\Delta U_A-\frac12x\cdot\nabla U_A
-\frac{1}{2(q-1)}U_A
+\partial_{x_1}\psi_A(U_A)=0
\qquad\text{in }\bbfR^N,
\end{equation}
with the Sobolev regularity stated in Theorem \ref{existence_I}. By
Theorem \ref{thm:uniform:estimates},
\[
\|U_A\|_\infty
\leqslant
\|V_A\|_\infty+\|U_0\|_\infty
\leqslant M(N,q,U_0),
\]
with $M$ independent of $A$. Choosing $A>M$, we have
\(
\psi_A(U_A)=|U_A|^{q-1}U_A.
\)
Hence $U:=U_A$ solves equation \eqref{ell_U_q}.

The Sobolev regularity gives
$V\in L^s(\bbfR^N)$ for some $1<s<N(q-1)$. Together with
$V\in L^\infty$, interpolation yields $V\in L^p$ for every
$p\geqslant s$. Lemma \ref{heat_L} gives
$U_0\in L^p$ whenever $p>N(q-1)$. Consequently,
\[
U\in L^p(\bbfR^N)
\qquad\text{for every }p \in (N(q-1),\infty].
\]

Define
\[
u(x,t)
=
t^{-\frac{1}{2(q-1)}}
U\left(\frac{x}{\sqrt t}\right).
\]
A change of variables in the weak formulation of \eqref{ell_U_q}
shows that $u$ is a distributional solution of
\[
u_t-\Delta u+\partial_{x_1}\bigl(|u|^{q-1}u\bigr)=0
\qquad\text{in }\bbfR^N\times(0,\infty).
\]

It remains to identify the initial trace. By the homogeneity of
$u_0$,
\[
e^{t\Delta}u_0(x)
=
t^{-\frac{1}{2(q-1)}}
U_0\left(\frac{x}{\sqrt t}\right),
\]
and therefore
\begin{equation}\label{u:decomposition}
u(x,t)
=
e^{t\Delta}u_0(x)
+
t^{-\frac{1}{2(q-1)}}
V\left(\frac{x}{\sqrt t}\right).
\end{equation}

Fix $1\leqslant r<N(q-1)$ and a bounded set
$K\subset\bbfR^N$. By the Sobolev regularity of $V$ and the
corresponding embeddings, there exists
\[
r\leqslant s<N(q-1)
\qquad\text{such that}\qquad
V\in L^s(\bbfR^N).
\]
Indeed, when $q\leqslant1+2/N$ this follows from the
$W^{1,p}$-regularity in Theorem \ref{existence_I}, whereas for
$q>1+2/N$ it follows from $V\in W^{1,2}(\bbfR^N)$, Sobolev
embedding and interpolation; for $N\geqslant3$ we use
$q<N/(N-2)$.

By scaling and H\"older's inequality,
\begin{align*}
\left\|
t^{-\frac{1}{2(q-1)}}
V\left(\frac{\cdot}{\sqrt t}\right)
\right\|_{L^r(K)}
&=
t^{\frac{N}{2r}-\frac{1}{2(q-1)}}
\|V\|_{L^r(K/\sqrt t)}
\\
&\leqslant
|K|^{\frac1r-\frac1s}
t^{\frac{N}{2s}-\frac{1}{2(q-1)}}
\|V\|_{L^s(\bbfR^N)}.
\end{align*}
Since $s<N(q-1)$, the exponent of $t$ is positive, and hence
\begin{equation}\label{V:initial}
t^{-\frac{1}{2(q-1)}}
V\left(\frac{\cdot}{\sqrt t}\right)
\longrightarrow0
\qquad\text{in }L^r(K)
\end{equation}
as $t\to0$.

Finally, $u_0\in L^r_{\rm loc}(\bbfR^N)$ for
$r<N(q-1)$, so
\[
e^{t\Delta}u_0\longrightarrow u_0
\qquad\text{in }L^r_{\rm loc}(\bbfR^N)
\]
as $t\to0$. Combining this with \eqref{u:decomposition} and
\eqref{V:initial}, we obtain
\[
u(\cdot,t)\longrightarrow u_0
\qquad\text{in }L^r_{\rm loc}(\bbfR^N)
\]
for every $1\leqslant r<N(q-1)$. This completes the proof.


\appendix

\section{Heat equation}
For convenience, this appendix collects the standard heat-equation
properties used in the paper.

Let $u(x,t)=e^{t\Delta}u_0(x)$ solve the linear Cauchy problem
\begin{align}\label{heat}
u_t&=\Delta u, && x\in\bbfR^N ,\quad t>0\\
u(x,0)&=u_0(x), &&x\in \bbfR^N.\nonumber 
\end{align}
It is given by convolution with the Gauss--Weierstrass kernel:
\begin{equation}\label{u:A}
u(x,t)=\int_{\bbfR^N}G(x-y,t)u_0(y)\,\ud y,
\qquad
G(x,t)=\frac{e^{-|x|^2/(4t)}}{(4\pi t)^{N/2}}.
 \end{equation}
Throughout this appendix, $u_0$ satisfies \eqref{u0}.

\begin{lemma}\label{lem:heat:1}
Assume $q>1+1/N$. The function \eqref{u:A} with initial datum
\eqref{u0} is smooth for $t>0$, bounded in $x$, and has the
self-similar form
\begin{equation}\label{self:u:A}
u(x,t)= t^{-\frac{1}{2(q-1)}}U_0\left(\frac{x}{\sqrt{t}}\right), 
\end{equation}
where 
\begin{equation}\label{U0_form}
U_0(x)\equiv(e^\Delta u_0)(x)=\int_{\bbfR^N} G(x-y,1)u_0(y)\ud y
\end{equation}
solves the elliptic equation
\begin{equation}\label{eq:U0:A}
-\Delta U_0-\tfrac{1}{2}x\cdot\nabla U_0
-\tfrac{1}{2(q-1)}U_0=0
\qquad\text{in }\bbfR^N.
\end{equation}
\end{lemma}

\begin{proof}
The change of variables $z=y/\sqrt t$ in \eqref{u:A} gives
\eqref{self:u:A}. Substitution of this formula into the heat equation gives
\eqref{eq:U0:A}.
\end{proof}

The following estimates are standard; see, for example,
\cite[Sec.~4]{B09}.

\begin{lemma}\label{heat_L}
Let $U_0=e^\Delta u_0$ be defined by equation \eqref{U0_form} with initial datum 
$u_0$ satisfying \eqref{u0}, and assume $q>1+1/N$. There is a
constant $C_0=C_0(q,N,\|C\|_\infty)>0$ such that
\begin{equation*}
\left|e^{\Delta }u_0(x) \right| +\left|\nabla e^{\Delta }u_0(x)\right| \leqslant \frac{C_0}{1+|x|^{\frac{1}{q-1}}}\quad \text{for all}\quad x\in \bbfR^N.
\end{equation*}
In particular,
$ e^{\Delta }u_0 \in L^p(\bbfR^N)$
and
$\nabla e^{\Delta }u_0 \in L^p(\bbfR^N)$
for each  $p\in (N(q-1), \infty ]$. 
If, moreover, $u_0$ is a $C^1$-function outside of the origin then
\begin{equation*}
\left|\nabla e^{\Delta }u_0(x)\right| \leqslant \frac{C_1}{1+|x|^{\frac{1}{q-1}+1}}\quad \text{for all}\quad x\in \bbfR^N.
\end{equation*}
\end{lemma}

\begin{proof}
This is a direct consequence of the explicit formula \eqref{U0_form}, see, \textit{e.g.}, {\cite[Lemma 1]{B09}} for a detailed proof.
\end{proof}

For more regular initial data, $e^\Delta u_0(x)$ behaves like $u_0(x)$ for large $|x|$, as made precise in the following lemma.

\begin{lemma}[{\cite[Lemma 2]{B09}}]\label{L:A3}
Assume, in addition, that $u_0\in C^1(\bbfR^N\setminus\{0\})$ and
\[
\sup_{x\neq0}|x|^{\frac{1}{q-1}+1}|\nabla u_0(x)|<\infty.
\]
Then
\begin{equation*}
  e^\Delta u_0(x)=u_0(x)+O\left(|x|^{-\frac{1}{q-1}-1}\right) \quad \text{as} \quad |x|\to +\infty.
\end{equation*}
\end{lemma}

\begin{remark}
Under the assumptions of Lemma \ref{L:A3}, we have $e^\Delta u_0 \notin L^p(\bbfR^N)$ for $p\in[1,N(q-1)]$, provided that $u_0$ has the corresponding nonintegrability at infinity.
\end{remark}


\end{document}